\documentclass{amsart}
\usepackage{graphicx} % Required for inserting images

\usepackage{amsmath}
\usepackage{amssymb}
\usepackage{amsthm}

\usepackage{dsfont}

\usepackage{esint}

\usepackage{mathrsfs}

\theoremstyle{theorem}
\newtheorem{theorem}{Theorem}[section]

\theoremstyle{definition}
\newtheorem{definition}[theorem]{Definition}

\theoremstyle{definition}

\theoremstyle{definition}

\theoremstyle{theorem}
\newtheorem{lemma}[theorem]{Lemma}

\theoremstyle{theorem}

\theoremstyle{theorem}

\theoremstyle{corollary}
\newtheorem{corollary}[theorem]{Corollary}

\theoremstyle{remark}
\newtheorem{remark}[theorem]{Remark}

\DeclareMathOperator{\tr}{tr}

\usepackage{polynom}

\DeclareMathOperator{\End}{End}

\DeclareMathOperator{\Hom}{Hom}

\DeclareMathOperator{\ind}{ind}

\DeclareMathOperator{\ch}{ch}
\DeclareMathOperator{\Td}{Td}

\DeclareMathOperator{\vol}{vol}

\DeclareMathOperator{\delbar}{\overline{\partial}}

\usepackage{comment}
\usepackage{appendix}

\usepackage{tikz-cd}
\usetikzlibrary{matrix,arrows,decorations.pathmorphing}

\title{Compact multi-monopole moduli spaces on K\"ahler surfaces}
\author{Ollie Thakar}
\address{\parbox{\linewidth}{Department of Mathematics, Harvard University, Massachusetts, 02138}}
\email{othakar@math.harvard.edu}

\begin{document}

\maketitle

\begin{abstract}
    We consider the multi-spinor Seiberg-Witten equations on a 4-manifold in a general setting, where we allow the auxiliary bundle to be non-trivial. We show that when the metric is K\"ahler and the auxiliary connection is unobstructed anti-self-dual Yang-Mills, the moduli space of solutions is compact for all parameters in some open neighborhood of the given metric and auxiliary connection. We use this fact to compute a count of solutions on some examples of K\"ahler surfaces. These computations have a surprising purely algebro-geometric corollary which may be of independent interest.
\end{abstract}

\section{Introduction}

The multi-spinor Seiberg-Witten equations have come to the attention of geometers recently as an example of a gauge-theoretic system of equations with a non-compact moduli space of solutions \cite{Bryan-Wentworth, Dedushenko-Gukov-Putrov, Taubes16}, a tool to understand $\mathbb{Z}/2$-harmonic spinors \cite{Doan-Walpuski}, and an ingredient in a putative enumerative invariant of $G_2$ manifolds \cite{haydys-speculation, walpuski-speculation}.

In this paper, we will define a count from the multi-spinor Seiberg-Witten equations on 4-manifolds, depending on several choices of data associated with these equations, and compute the value of this quantity for certain K\"ahler surfaces. The key idea allowing us to perform these computations is as follows: while the multi-spinor Seiberg-Witten moduli spaces are not compact in general, when the auxiliary connection and metric are in a small open neighborhood of an unobstructed anti-self-dual Yang-Mills connection and K\"ahler metric, the moduli space is compact.

\subsection{The multi-spinor Seiberg-Witten Equations}
The multi-spinor Seiberg-Witten equations are gauge-theoretic equations on a closed, smooth, oriented 4-manifold $X$ with the following additional data:

\begin{itemize}
    \item A Riemannian metric $g$
    \item A $SU(N)$-bundle $E\to X$ with $k:= c_2(E)\geq0.$
    \item A $SU(N)$-compatible connection $B\in\mathscr{A}(E)$
    \item A spin$^c$-structure $\mathfrak{s}$ with spinor bundles $W^\pm\to X,$ determinant bundle $S,$ and with Clifford multiplication $\rho:T^*X\otimes \mathbb{C}\to\Hom(W^\pm, W^\mp)$
    \item A self-dual 2-form $\eta\in\Omega^+_X$ we will use to perturb the equations
\end{itemize}

The equations ask for a unitary connection $A'$ on the determinant bundle $S$, and a section $\Phi\in\Gamma_X(W^+\otimes E)$ which will be referred to as a \emph{twisted spinor.} There is a unique spin$^c$ connection on $W^\pm$ which induces the connection $A'$ on the determinant bundle, and we will also label this connection as $A'.$ The equations are: \begin{equation}\label{msw}\begin{cases}
    D_{A'\otimes B}^+\Phi &= 0 \\
    \rho(F_{A'}^++\sqrt{-1}\eta) -\mu(\Phi, \Phi) &= 0, \\
\end{cases}\end{equation} where $\mu:(W^+\otimes E)\times\overline{(W^+\otimes E)}\to \mathfrak{sl}(W^+)$ is a bilinear map given by: $$\mu(\Phi, \Psi) := \tr_E(1-\tr_{W^+})(\Phi\otimes \Psi^*).$$

These equations are preserved by the action of the gauge group $\mathcal{G} = \text{Map}(X, S^1)$ given as $u\cdot (A', \Phi) = (A' - u^{-1}du, u\cdot \Phi).$ The moduli space of such solutions, modulo unitary gauge, is to be called $\mathcal{M}(X, g, B, \eta).$

\begin{remark}
    Solutions to the multi-spinor Seiberg-Witten equations will occasionally be referred to as \emph{multi-monopoles.}
\end{remark}

\subsection{Outline and main theorems}

We will first define the \emph{Seiberg-Witten function}, which takes the form of a locally constant map from an open subset $\mathscr{U}_{cpt}$ in the space of choices of $(g, B, \eta)$ to the integers. Informally, this is a count of the number of solutions to the multi-monopole equations corresponding to this data. The subset $\mathscr{U}_{cpt}$ consists of those $(g, B, \eta)$ such that the moduli space $\mathcal{M}(X, \mathbf{p})$ is compact for all $\mathbf{p}$ sufficiently near to $(g, B, \eta)$. When the multi-spinor moduli space is compact and of its formal dimension, this function agrees with the usual definition of Seiberg-Witten invariants. We will state this loosely below, but the formal statement and definitions are found in Theorem \ref{local-invariance}.

\begin{theorem}
    Suppose that $b^+(X) \geq 2.$ Then, there exists a locally constant function $\mathscr{U}_{cpt}\to\mathbb{Z},$ denoted as $$\mathbf{p}\mapsto\text{SW}_E(X, \mathbf{p}),$$ which agrees with the evaluation of the appropriate power of the first Chern class of the canonical line bundle on the configuration space on the moduli space $\mathcal{M}(X, \mathbf{p})$, when $\mathcal{M}(X, \mathbf{p})$ is a transversely cut-out smooth manifold.
\end{theorem}

The remainder of this paper will focus on the case where $(X, g)$ is K\"ahler, with K\"ahler form $\omega,$ and $B$ has $F_B^{2,0} = 0,$ so that it defines a holomorphic structure $E_B$ on the bundle $E.$ In Section 3, we will find a holomorphic interpretation of the moduli space in this case.

In general, it seems somewhat intractable to prove that the space of parameters for which the moduli space is compact is open in the space of all parameters. It is already a highly subtle question to determine when t$\mathscr{U}_{cpt}$ is empty. However, In Section 4, we will partially address both of these concerns, by showing that when $(g, \omega)$ is K\"ahler and $B$ is unobstructed, $(g, B, 2r\omega) \in \mathscr{U}_{cpt}$. We state a version of the theorem here, but the strongest version of the theorem will be stated as Theorem \ref{strong_compact}.

\begin{theorem}\label{strong_compact}
    If $H^2(X; \End_0 E_B) = 0$ and $r\in\mathbb{R}_+$ is sufficiently large, then there exists some open neighborhood $\mathcal{U}$ of $\mathbf p = (g, B, 2r\omega)$ in the space $\text{Met}(M)\times \mathcal{A}(E)\times \Omega^2(M)$ such that for any $(g', B', \eta') \in \mathcal{U}$ the moduli space $\mathcal{M}(X, g', B', \eta')$ is compact, and so $\mathbf p \in \mathscr{U}_{cpt}.$
\end{theorem}

Taubes has proved a beautiful compactness theorem for the multi-spinor Seiberg-Witten equations on a general Riemannian 4-manifold \cite{Taubes16}, which will play a central role in our strong compactness theorem. (See Section 4.2). In Section 5 we will show that the convergence in Taubes' theorem can be strengthened in the K\"ahler case. This strengthening culminates in the computation of the limit of the curvatures of a sequence of multi-monopoles when $g$ is K\"ahler and $F_B^{2,0} = 0$ (Theorem \ref{currents}.)

Finally, in Section 6, we will compute the value of the Seiberg-Witten function on certain K\"ahler surfaces. In Section 6, we will compute the value of the Seiberg-Witten function on certain K\"ahler surfaces. The local invariance of the Seiberg-Witten function allows us to extract the following surprising theorem (see Theorem \ref{algebraicgeometry}.) 

\begin{theorem}
    Let $X$ be a K\"ahler surface with $h^1(\mathscr{O}_X) = 0, h^2(\mathscr{O}_X)>0.$ Let $\mathcal{N}$ be an irreducible component of the moduli space of stable vector bundles on $X$ with holomorphically trivial determinant. Let $\mathscr{L}\to X$ be a holomorphic line bundle.
    
    If there is some $\mathscr{F}_0\in\mathcal{N}$ satisfying the following three conditions: \begin{itemize} 
        \item $H^2(\End_0 \mathscr{F}_0) = 0$
        \item $\chi(\mathscr{F}_0\otimes \mathscr{L})>h^0(\mathscr{F}_0\otimes \mathscr{L})>0$ 
        \item $\chi(\mathscr{F}_0\otimes \mathscr{L})>h^2(\mathscr{O}_X)$
    \end{itemize} then for any $\mathscr{F}\in\mathcal{N}$ with $H^2(\End_0 \mathscr{F}) = 0$, either $$h^0(\mathscr{F}\otimes\mathscr{L}) = h^0(\mathscr{F}_0\otimes \mathscr{L})$$ or $$h^0(\mathscr{F}\otimes \mathscr{L})=2\chi(\mathscr{F}\otimes \mathscr{L}) - h^2(\mathscr{O}_X)-h^0(\mathscr{F}_0) .$$
\end{theorem}

\subsection{Historical background}
The 3-manifold version of these equations has attracted notable attention. For instance, Doan has defined an enumerative invariant from this equation in \cite{Doan} that is the direct inspiration for our construction, which has been computed in some special cases \cite{Doan, Thakar, Wilson}.

Multi-spinor moduli spaces on 4-manifolds appear much less common in the literature. They were first considered n the case where $E$ and $B$ are trivial in \cite{Bryan-Wentworth}. In that work, Bryan and Wentworth found an algebro-geometric description of the moduli spaces on a K\"ahler surface; in this paper, we perform the analogous algebro-geometric description in our more general setting. This, along with work on generalized Seiberg-Witten equations by Taubes \cite{Taubes16}, are the key tools allowing us to compute the value of our count on many K\"ahler surfaces. More recently, these equations have been considered from a physical point of view in \cite{Nakajima}, and \cite{Dedushenko-Gukov-Putrov} defined an invariant associated to multi-spinor moduli spaces with $E$ and $B$ trivial in a very different way; it would be interesting to compare their construction with ours. 

\subsection{Notation and Conventions}

Hermitian forms are conjugate linear in the second factor.

The space $\mathscr{A}(E)$ denotes those connections on $E$ compatible with the $SU(N)$-structure on $E.$

\subsection{AI Use Statement}

Before June 2026, the author attempted unsuccessfully to bounce ideas off of AI. After June 2026, the author began a boycott of AI in mathematics and has not used it in the paper since.

\subsection{Acknowledgments}

I would like to thank my advisor Peter Kronheimer as well as Aleksander Doan, Clifford Taubes, and Elie Belkin for sharing invaluable wisdom, and Greg Parker for providing insightful comments on an earlier draft. This project was made under the support of the Simons Collaboration on New Structures in Low-Dimensional Topology.

\section{The multi-spinor Seiberg-Witten function}

The main goal of this section is to delineate under which conditions our count may be defined and prove very basic properties about it. This can be viewed as the 4-dimensional version of the putative multi-spinor Seiberg-Witten invariant of 3-manifolds explored in, for instance, \cite{Doan}. 

\subsection{Geometric preliminaries}

So, fix a closed, smooth, oriented 4-manifold $X$ with a spin$^c$-structure $\mathfrak{s}$ and an $SU(N)$-bundle $E\to X$. We will define the \emph{space of parameters} $\mathscr{P}$ to be the subset: $$\mathscr{P} \subseteq \text{Met}(X)\times \mathscr{A}(E)\times \Omega^2_X$$ consisting of those $(g, B, \eta)$ such that $\eta$ is self-dual with respect to $g.$ We define the subset $\mathscr{P}_{cpt}\subseteq \mathscr{P}$ to be those $\mathbf{p} = (g, B, \eta) \in \mathscr{P}$ for which the moduli space $\mathcal{M}(X, \mathbf p)$ is compact. We will define the subset $\mathscr{U}_{cpt}\subseteq \mathscr{P}_{cpt}$ to be the interior of $\mathscr{P}_{cpt},$ i.e. those $\mathbf{p}\in \mathscr{P}_{cpt}$ such that for a sequence $\mathbf{p}_i\to\mathbf{p}$ in $C^\infty,$ we have $\mathbf{p}_i\in \mathscr{P}_{cpt}$ for all sufficiently large $i.$

\begin{remark}
    A priori it is far from obvious that $\mathscr{U}_{cpt}$ is non-empty; the following sections will provide reasonable conditions guaranteeing that $\mathscr{U}_{cpt}$ is non-empty when $X$ is a K\"ahler surface.
\end{remark}

We recall the usual set-up necessary to describe Seiberg-Witten moduli spaces and to define Seiberg-Witten invariants. Define the \emph{configuration space} $\mathscr{C} = \mathscr{A}(S)\times \Gamma(W^+\otimes E),$ and the space $\mathscr{B} = \mathscr{C}/\mathcal{G},$ so that for each $\mathbf p\in\mathscr{P},$ the moduli space $\mathcal{M}(X, \mathbf p)\subset \mathscr{B}.$ Choosing a basepoint $x\in X,$ we may define the \emph{framed gauge group} $\mathcal{G}_0\subseteq \mathcal{G}$ to be those $u\in \text{Map}(X, S^1)$ such that $u(x) = 1,$ so that $\mathscr{C}/\mathcal{G}_0$ is the total space of a principal $U(1)$-bundle over $\mathscr{B}.$ Let $\mu\in H^2(\mathscr{B}; \mathbb{Z})$ represent the first Chern class of this bundle (we will abuse notation and let $\mu$ also represent the first Chern class of the restriction of this bundle to each moduli space $\mathcal{M}(X, \mathbf p)\subset \mathscr{B}.$) 

\subsubsection{The deformation complex}

At each $[A', \Phi]\in \mathcal{M}(X, \mathbf p),$ the real-analytic structure of the moduli space is governed by the \emph{deformation complex} $\mathcal{C}_{A', \Phi},$ the chain complex $\mathcal{C}^0\to\mathcal{C}^1\to\mathcal{C}^2$ defined as follows:

$$\Omega^0_X(i\mathbb{R})\xrightarrow{G_{[A', \Phi]}} \Omega^1_X(i\mathbb{R})\times\Gamma(W^+\otimes E)\xrightarrow{F_{A', \Phi}} \Omega^+_X(i\mathbb{R})\times \Gamma(W^-\otimes E),$$ where $$G_{A', \Phi}(f) = (-df, f\Phi)$$ is the linearization of the gauge group action, and $$F_{A', \Phi}(A', \Phi) = \left(\sqrt{-1}d^+a - 2\sqrt{-1}~\text{Im} \mu(\Phi, \phi), D_{A'\otimes B}^+\phi + \rho(a)\Phi\right)$$ is the linearization of Equation (\ref{msw}). For $j=0,1,2,$ let $H^j_{A', \Phi}$ denote the degree $j$ cohomology of the complex $\mathcal{C}_{A', \Phi}.$ 

We say a solution $(A', \Phi)$ is \emph{reducible} if $H^0_{A', \Phi}\neq 0$ (which is easily seen to be equivalent to $\Phi = 0$) and otherwise it is irreducible. Likewise, we say a solution $(A', \Phi)$ is \emph{obstructed} if $H^2_{A', \Phi}\neq 0,$ and otherwise it is unobstructed. We say the moduli space $\mathcal{M}(X, \mathbf p)$ is \emph{unobstructed} if each solution $(A', \Phi)$ is unobstructed, and we say the moduli space $\mathcal{M}(X, \mathbf p)$ is \emph{regular} if each solution is both unobstructed and irreducible. The \emph{virtual dimension} of the moduli space is the index of the complex $\mathcal{C}.$

In general, near each $[A', \Phi]\in\mathcal{M}(X, \mathbf p)$ corresponding to a choice of slice for the gauge group action near $(A', \Phi),$ there is a Kuranishi map $\kappa:H^1_{A', \Phi}\to H^2_{A', \Phi}$ with $\kappa(0) = d\kappa(0),$ such that as a real analytic space a neighborhood of $[A', \Phi]$ in $\mathcal{M}(X, \mathbf p)$ is equivalent to a neighborhood of $\kappa^{-1}(0)\in H^1_{A', \Phi}$ (see, for instance, \cite[Chapter 4]{DK}.)

\subsection{Local invariance}

When $\mathcal{M}(X, \mathbf p)$ is regular, the space $\mathcal{M}(X, \mathbf p)$ is a smooth manifold of dimension equal to the virtual dimension $d$ determined by Atiyah-Singer index theorem. When moreover $\mathbf{p} \in \mathscr{P}_{cpt}$, we may define a ``count of solutions'' to be the integral: $$\int_{\mathcal{M}(X, \mathbf p)} \mu^{d/2}.$$ The main theorem of this section will be the 4-dimensional analog of \cite[Theorem 1.3]{Doan}: 
\begin{theorem}\label{local-invariance}
    Suppose that $b^+(X) \geq 2.$ Then, there exists a locally constant function $\mathscr{U}_{cpt}\to\mathbb{Z},$ denoted as $$\mathbf{p}\mapsto\text{SW}_E(X, \mathbf{p}),$$ which enjoys the property that for generic $\mathbf p \in \mathscr{U}_{cpt}$ we have: $$\text{SW}_E(X, \mathbf{p}) = \int_{\mathcal{M}(X, \mathbf p)} \mu^{d/2},$$ where $d = \dim \mathcal{M}(X, \mathbf p).$
\end{theorem}

The proof of this theorem is very familiar from Seiberg-Witten theory, so we provide a brief outline and references for most of the standard arguments: we will first show that the moduli space is generically unobstructed and contains no reducible solutions; then, the theorem will follow when we show that moduli spaces corresponding to generic 1-parameter families of parameters provide cobordisms of the moduli spaces corresponding to the endpoints of the families.

\subsubsection{Generic Unobstructedness}
We will prove a slight strengthening of the statement that generic parameters result in unobstructed moduli spaces:

\begin{lemma}\label{generic}
    For each fixed Riemannian metric $g\in \operatorname{Met}(X),$ there is a residual subset $\mathscr{R}_g\subset \mathscr{A}(S)\times\Omega^+_X$ such that the irreducible part of the moduli space $\mathcal{M}(X, g, B, \eta)$ is unobstructed whenever $(B, \eta)\in \mathscr{R}_g.$
\end{lemma}

\begin{proof}
    Fix $g$. We first take parameters $(B, \eta)$ in the Sobolev space $L^2_k,$ for fixed $k>>0.$ Fix an irreducible multi-monopole $(A', \Phi).$ For such parameters, we consider the differentials of the Seiberg-Witten function and the gauge fixing condition. If these constitute a surjective map $$T_A\mathscr{A}(L)\times T_\Phi \Gamma(W^+\otimes E)\times T_B\mathscr{A}(E)\times T_\eta\Omega^+_X\to \Gamma(W^-\otimes E)\times i\Omega^+_X\times \widetilde\Omega^0_X$$ (with all sections taken in $L^2_k$), then the Sard-Smale theorem assures us that the irreducible part of the moduli space $\mathcal{M}(X, g, B, \eta)$ is unobstructed for a residual subset of $L^2_k$ parameters $(B,\eta)$. (Here, $\widetilde\Omega^0_X$ denotes those functions integrating to 0.) Passing from $L^2_k$ to $C^\infty$ is a standard argument based on elliptic bootstrapping and an idea of Taubes; see \cite[Proposition 2.19]{Doan}.

    The map in question is: $$(A', \Phi', b, \eta') \mapsto (\rho(b)\Phi+\rho(a)\Phi + D_{AB}\Phi', d^+a -2\rho^{-1}(\text{Im}\mu(\Phi, \Phi'))+\sqrt{-1}\eta', d^*a)$$ For fixed $b, \Phi',$ by varying $a, \eta'$ the projection onto the last two coordinates is surjective. To show surjectivity of the whole map, it suffices to show that for fixed $a, \eta',$ if we vary $b$ and $\Phi'$ we achieve all values of the projection to the first coordinate. This follows from \cite[Lemma 2.20]{Doan} together with the unique continuation theorem for solutions to the twisted Dirac equation, as in \cite[Proposition 2.19]{Doan} or \cite[Transversality Theorem 1]{Moore}.
\end{proof}

\subsubsection{Generic Irreducibility}
Note also the following result, which is standard in Seiberg-Witten theory (such as \cite[Proposition 2.21]{Doan},   \cite[Proposition of Section 1.9]{Moore}).
\begin{lemma}
    If $b^+(X)\geq 1$, for a generic parameter the moduli space is regular, and if $b^+(X)\geq 2$, for a generic 1-parameter family of parameters the moduli space is irreducible.
\end{lemma}  

\subsubsection{Orientations}

Also, just as in ordinary Seiberg-Witten theory, by specifying an orientation of $H^1(X; \mathbb{R})\oplus H^+(X; \mathbb{R})$ we get a canonical orientation on the moduli space $\mathcal{M}(X, \mathbf{p})$, when it is irreducible and unobstructed, and this orientation is invariant in smooth 1-parameter families.

\begin{proof}[Proof of Theorem \ref{local-invariance}]
    Lemma \ref{generic} holds true for generic 1-parameter families of parameters $(g, B, \eta)$ by the same argument with the Sard-Smale theorem as in the above proof. 
    
    By the proof of \cite[Theorem 1.3]{Doan}, our desired local invariance statement follows immediately from generic unobstructedness, generic irreducibility, and our discussion of orientations.
\end{proof}

\subsection{Zariski smooth moduli spaces}

Many times in this paper, we will consider situations in which the moduli space $\mathcal{M}(X, \mathbf p)$ is a smooth manifold of dimension \emph{not} equal to its virtual dimension. In one such case, we may nonetheless directly compute the value $SW_E(X, \mathbf p)$ from the moduli space together with one piece of extra data, which is the \emph{obstruction bundle}.

So, say that $\mathcal{M}(X, \mathbf p)$ is \emph{Zariski smooth} if at each $[A', \Phi]\in \mathcal{M}(X, \mathbf p)$ there is a Kuranishi map which is identically zero. In this case, we cite the following facts from \cite[Section 2]{Doan}: Each connected component of the moduli space $\mathcal{M}(X, \mathbf p)$ is a smooth manifold and the tangent space at each $[A', \Phi]$ may be identified with $H^1_{A', \Phi}.$ Moreover, the spaces $H^2_{A', \Phi}$ fit together to form a vector bundle over each component of $\mathcal{M}(X, \mathbf p)$ which shall be called the \emph{obstruction bundle} and denoted as $\mathfrak{O}\to \mathcal{M}(X, \mathbf p).$ As in ordinary Seiberg-Witten theory, the bundle $\Lambda^{\text{top}} T\mathcal{M}\otimes (\Lambda^{\text{top}} \mathfrak{O})^*$ is the determinant bundle of the index bundle of the linearized Seiberg-Witten equations with gauge fixing over the moduli space $\mathcal{M}.$ Hence, it has a canonical orientation, once we specify an orientation for $H^1(X;\mathbb{R})\oplus H^+(X; \mathbb{R})$ (see for instance \cite[Transversality Theorem 2]{Moore}.) Hence, if $\mathcal{M}$ is orientable, so is $\mathfrak{O},$ and moreover there is a canonical relative orientation of $\mathcal{M}$ and $\mathfrak{O}$.

\begin{remark}
    In fact, $\mathfrak{O}\to \mathcal{M}(X, \mathbf p)$ actually represents a vector bundle of a potentially different rank over each component of $\mathcal{M}(X, \mathbf p).$ 
\end{remark}

The following theorem from Friedman-Morgan allows us to compute the value of the Seiberg-Witten function from a Zariski smooth moduli space, and will be our main computational tool.

\begin{theorem}[{\cite[Theorem 3.1]{FM}}]\label{zariski-smooth}
    Suppose $\mathbf{p} \in \mathscr{U}_{cpt}$ and that $\mathcal{M}(X, \mathbf p)$ is irreducible, orientable, and Zariski smooth. Then, $$SW_E(X, \mathbf p) = \int_{\mathcal{M}(X, \mathbf p)} e(\mathfrak{O})\smile\mu^{d/2},$$ where $d$ is the virtual dimension of $\mathcal{M}(X, \mathbf p).$
\end{theorem}

\section{The equations on a K\"ahler surface}
In this section, we will describe the moduli space $\mathcal{M}(X, g, B, \eta)$ when $X$ is K\"ahler and $B$ has been chosen appropriately. We will find that, much like the ordinary Seiberg-Witten moduli space on a K\"ahler surface, there exists an algebro-geometric interpretation. Moreover, there exists an algebro-geometric interpretation of limiting configurations appearing as limits of rescaled sequences of solutions, and which may be thought of as $\mathbb{Z}/2$-harmonic spinors (see \cite{Taubes16}.)

Let $X$ be a compact K\"{a}hler surface with K\"{a}hler form $\omega\in\Omega^+.$ Then, $X$ has a canonical spin$^c$ structure $\mathbb{S}^\pm$ with $\mathbb{S}^+ = \Omega^0\oplus\Omega^{0,2}$ and $\mathbb{S}^- = \Omega^{0,1}.$ We now will choose $L$ a hermitian line bundle such that $W^\pm = \mathbb{S}^\pm\otimes L.$ Define the perturbation form as $\eta = 2r \omega$ for $r$ very large. Let $\nabla_0$ be the unitary connection on the canonical bundle of $X$ induced from the Levi-Civita connection, and $\nabla_0'$ be the induced connection on $K^{-1}.$

The correspondence between unitary connections on $L$ and those on $S$, taking a connection $A$ on $L$ to $A':=A^2\otimes \nabla_0'$ on the determinant bundle $S,$ is bijective. Hence, we may write the equations in terms of the connection $A$ on $L,$ which in fact will simplify the notation. So, following \cite{Bryan-Wentworth}, but making the appropriate generalizations for the fact that the connection $B$ is not trivial, the twisted spinor $\Phi = (\alpha, \beta)$ in accordance with the splitting $W^+ = L\oplus (K_X^{-1}\otimes L),$ and so Equations \ref{msw} become the following, for $A\in\mathcal{A}(L), \alpha \in \Omega^0(L\otimes E), \beta\in\Omega^{0,2}(L\otimes E)$:

\begin{equation}\label{holomsw}\begin{cases}
    \delbar_{AB}\alpha &= -\delbar_{AB}^*\beta \\
    F_A^{2,0} &= \tr(\alpha\otimes\beta^*) \\
    -\sqrt{-1}\Lambda F_A &= \frac14\left(|\alpha|^2-|\beta|^2\right) - \frac12\sqrt{-1}\Lambda F_{\nabla_0} - r.\\
\end{cases}\end{equation}

\begin{lemma}\label{holomorphic}
    If $F_B^{2,0} = 0,$ then any solution of the above equations must have: $$F_A^{2,0} = \delbar_{AB}\alpha = \delbar^*_{AB}\beta = 0.$$
\end{lemma}

\begin{proof}
    Applying the operator $\delbar_{AB}$ to the first equation, we get: $$(F_A^{2,0}+F_B^{2,0})\alpha = \delbar_{AB}\delbar_{AB}\alpha = -\delbar_{AB}\delbar_{AB}^*\beta.$$ Pairing with $\beta$ and integrating both sides, we see by $F_A^{2,0} = 0$ and integration by parts that: $$\int_X\langle F_A^{2,0}\alpha, \beta\rangle = -\int_X |\delbar_{AB}^*\beta|^2.$$ By the second equation, we have: $$ \langle F_A^{2,0}\alpha, \beta\rangle = \langle \tr(\alpha\otimes\beta^*)\alpha, \beta\rangle = |\tr(\alpha\otimes\beta^*)|^2,$$ so now we may immediately conclude that $\delbar_{AB}^*\beta = \tr(\alpha\otimes\beta^*) = 0,$ and the remaining claims in the lemma statement follow from the first two equations.
\end{proof}

\subsection{Holomorphic Triples}

For the remainder of this section, we will assume that $F_B^{2,0} = 0$ and use $\mathscr{E}$ or occasionally $E_B$ to denote the holomorphic bundle determined by the connection $B$ on $E.$ Under this assumption, we will find a holomorphic interpretation of the moduli space. To do this, we must first define the holomorphic configuration spaces.

The \emph{holomorphic configuration space} $\mathscr{C}_\mathbb C$ is the space $\mathscr{C}_\mathbb C = \mathscr{A}(L)\times\Omega^0(E\otimes L)\times \Omega^{0,2}(E\otimes L).$ We define the \emph{complex gauge group} to be $$\mathcal{G}_\mathbb C := \text{Map}(X, \mathbb{C}^*),$$ and this acts on the holomorphic configuration space by: $$u\cdot(A, \alpha, \beta) = (A+\overline u^{-1}\partial\overline u-u^{-1}\delbar u, u\alpha, \overline{u}^{-1}\beta).$$ Define the space $\mathscr{B}_\mathbb C$ to be the quotient $\mathscr{C}_\mathbb C/\mathcal{G}_\mathbb C.$

\begin{definition}
    For a connection $B\in\mathscr{A}(E)$ with $F_B^{2,0} = 0$, we define a $E_B$-\emph{holomorphic triple} to be the data of $(A, \alpha, \beta)$ where: \begin{itemize}   \item $A\in\mathscr{A}(L)$ has $F_A^{2,0}=0$. So, it defines a  holomorphic bundle $\mathscr{L}\in \text{Pic}(X)$ whose underlying smooth bundle is $L,$ and 
    \item $\alpha\in H^0_X(\mathscr{E}\otimes \mathscr{L}),$ and \item $\beta^*\in H^0_X(\mathscr{E}^*\otimes \mathscr{L}^*\otimes K_X),$ 
    \end{itemize} satisfying $\tr(\alpha\otimes \beta^*) = 0$ and $\alpha$ not identically zero.
\end{definition}

Define the \emph{moduli space of holomorphic triples} $\mathcal{M}_\mathbb{C}(X, \mathscr{E}, L)$ to be the subset of $\mathscr{B}_\mathbb C$ consisting of those $[A, \alpha, \beta]\in \mathscr{B}_\mathbb C$ such that $(A, \alpha, \beta)$ is an $\mathscr{E}$-holomorphic triple.

\begin{remark}\label{picard}
    By forgetting the sections, the moduli space $\mathcal{M}_\mathbb{C}(X, \mathscr{E}, L)$ maps to the Picard group of $X$. The fiber over each $\mathscr{L}\in\text{Pic}(X)$ is the quotient of the space $\left\{\tr(\alpha\otimes \beta^*) = 0 | (\alpha, \beta^*)\in H^0_X(\mathscr{E}\otimes\mathscr{L})\times H^0_X(\mathscr{E}^*\otimes\mathscr{L}^*\otimes K_X)\right\}$ by the action of $\mathbb{C}^*$ as $\lambda\cdot (\alpha, \beta^*) = (\lambda\alpha, \lambda^{-1}\beta^*).$
\end{remark}

The following theorem is the analog of \cite[Theorem 3.2]{Bryan-Wentworth} in our more general setting, and the proof follows similar lines:

\begin{theorem}\label{complex-moduli}
    If $F_B^{2,0}= 0,$ then the moduli space $\mathcal{M}(X, g, B, 2r\omega)$ is homeomorphic to the space of holomorphic triples $\mathcal{M}_\mathbb{C}(X, E_B, L)$.
\end{theorem}

\begin{proof}
    First, we define a map $f:\mathcal{M}(X, g, B, r\omega)\to \mathcal{M}_\mathbb{C}(X, E_B, L),$ by $$f(A, (\alpha, \beta)) := [A, \alpha, \beta].$$ By the first two equations of Equation (\ref{holomsw}) and Lemma \ref{holomorphic}, $[A, \alpha, \beta]$ is a holomorphic triple. The map $f$ is well-defined since for any gauge transformation $u\in \mathcal{G},$ evidently we have $f(u\cdot(A', \Phi)) = u\cdot f(A', \Phi)$ and $u\in \mathcal{G}\subseteq \mathcal{G}_\mathbb{C}$. Clearly $f$ is continuous.

    We will construct an inverse to $f$ as follows. A holomorphic triple $(A, \alpha, \beta)$ gives us a solution to the multi-spinor Seiberg-Witten equations if and only if it satisfies the third equation of (\ref{holomsw}): $$-\sqrt{-1}\Lambda F_A =\frac14\left(|\alpha|^2-|\beta|^2\right) - \frac12\sqrt{-1}\Lambda F_{A_0} - r.$$ So, given a holomorphic triple $(A, \alpha, \beta)$ we wish to find a gauge transformation $u = e^h \in \mathcal{G}_\mathbb C$ such that $u(A, \alpha, \beta)$ solves the equation: $$-\sqrt{-1}\Lambda F_{uA} =\frac14\left(|u\alpha|^2-|\overline{u}^{-1}\beta|^2\right) - \frac12\sqrt{-1}\Lambda F_{A_0} - r.$$ 

    Using the K\"ahler identities, we compute that $\sqrt{-1}\Lambda F_{uA} = \sqrt{-1}\Lambda F_A + \Delta h,$ where $\Delta$ is the Hodge Laplacian. Hence, we must solve for $h$ satisfying: \begin{equation}\label{kw} \Delta h +e^{2h}(\frac14|\alpha|^2) - e^{-2h}(\frac14|\beta|^2) =\frac12 \sqrt{-1}\Lambda F_{A_0} - \sqrt{-1}\Lambda F_A + r.\end{equation} Setting $$P = \frac14|\alpha|^2, Q = \frac14|\beta|^2, w = \frac12 \sqrt{-1}\Lambda F_{A_0} - \sqrt{-1}\Lambda F_A + r,$$ we may re-write this as a variant of a Kazdan-Warner equation: $$\Delta h + Pe^{2h}-Qe^{-2h} = w.$$ Bryan and Wentworth show the following:
    
    \begin{lemma}[{\cite[Lemma 3.4]{Bryan-Wentworth}}] On any compact Riemannian manifold, the equation: $$\Delta h + Pe^{2h}-Qe^{-2h} = w$$ admits a solution provided that $P\geq 0, Q\geq 0, \int (P-Q)>0,$ and $\int w >0.$ Moreover, it has at most one solution provided $P\geq 0$ and $Q\geq 0.$ 
    \end{lemma}
    
    In our case, the first two conditions are satisfied evidently, and so we are guaranteed that such a solution is unique; the third condition, as in \cite[Theorem 4.2]{Doan}, may be guaranteed by first performing a constant gauge transformation, and the fourth condition is guaranteed by taking $r$ sufficiently large (the integrals of the curvature terms are determined topologically by the Chern-Weil theory.)

    We then get a map $\mathcal{M}_\mathbb C(X, E_B, L)\to \mathcal{M}(X, \mathbf p)$ which takes $(A, \alpha, \beta)$ to $e^h(A, \alpha, \beta)$ where $h$ is the unique solution to Equation \ref{kw}, which is clearly an inverse to $f$. Such inverse is continuous by the argument of \cite[Proposition 4.7]{Doan}.
\end{proof}

\subsection{Real Analytic Structure of Holomorphic Moduli Space}

This subsection is devoted to understanding the real analytic structure of the holomorphic moduli space. Since it is the only case we need for our computations, we will do this in a slightly restricted setting, which is the setting in which all solutions to Equation \ref{holomsw} with $\alpha \not\equiv0$ have $\beta = 0$. Declare a parameter $\mathbf p$ to have Property (C) if this is true. For later, we record the following lemma, which is the primary motivation (and etymology) for this condition:

\begin{lemma}\label{property-c}
    If a parameter $\mathbf p$ has Property (C), then $\mathcal{M}(X, \mathbf p)$ is compact.
\end{lemma}

\begin{proof}
    For any solution $(A, (\alpha, \beta)),$ we may take the integral of the third equation of (\ref{holomsw}). By Chern-Weil theory: $$\pi\left(2c_1(L)-c_1(K_X)\right)\cdot[\omega] +r[\omega]^2 = ||\alpha||_{L^2}^2 - ||\beta||_{L^2}^2.$$ Since $\beta = 0,$ this gives us an \emph{a priori} bound on $||(\alpha, \beta)||_{L^2}^2,$ which guarantees the compactness of the moduli space $\mathcal{M}(X, \mathbf p)$ (there are many ways to see this; for instance, see Theorem \ref{taubes} in this paper for more details.)
\end{proof}

 The moduli space of holomorphic triples itself has a deformation theory, which we will describe by the \emph{holomorphic deformation complex} $\mathcal{C}_\mathbb C.$ At a solution $(A, \alpha, 0)$ we define this complex as follows: $$\Omega^0\xrightarrow{\delta^0}\Omega^{0,1}_X\oplus \Omega^0(L\otimes E)\oplus \Omega^{0,2}(L\otimes E)\xrightarrow{\delta^1} \Omega^{0,2}_X\oplus \Omega^{0,1}(L\otimes E),$$ where all sections of bundles and differential forms are taken with complex coefficients. The map $\delta^0$ is the linearization of the complex gauge group action (note the gauge group action looks different here because we consider a unitary connection as determined by its $(0,1)$-part): $$h\mapsto (-\overline\partial h, h\alpha, 0)$$ and the map $\delta^1$ is given by the linearization of the first two equations of Equation (\ref{holomsw}): $$(a, \alpha', \beta')\mapsto (\overline\partial a +\tr(\alpha^*\otimes \beta'), \delbar_{AB}\alpha' + \delbar^*_{AB}\beta' + a\alpha).$$ We will denote the $j$th cohomology of this complex by $H^j_{\mathbb C}$ or $H^j_{\mathbb C, A, \alpha}$. 

The following theorem comes immediately from the argument of \cite[Theorem 2.1]{FM}:

\begin{theorem}\label{deformation-theories}
    Suppose $(A, \alpha, 0)$ is an $E_B$-holomorphic triple. Then for $j = 1, 2$ there are isomorphisms of real vector spaces $H^j_{A, (\alpha, 0)}\cong H^j_{\mathbb{C}, A, \alpha}$ which commute with the action of the gauge group $\mathcal{G}.$ 
\end{theorem}

    %First, we show: \begin{lemma} $\delta^1(a, \alpha', \beta') = 0$ implies $\beta' = 0$ and $\delbar a = 0.$ \end{lemma}

    %\begin{proof} Since $F_A^{2,0} = F_B^{2,0} = 0,$ applying the operator $\delbar_{AB}$ to the equation $\delbar_{AB}\alpha' + \delbar^*_{AB}\beta' + a\alpha = 0$ gives us: $$0 = \delbar_{AB}\delbar^*_{AB}\beta' + (\delbar a)\alpha = \delbar_{AB}\delbar^*_{AB}\beta' + \tr(\alpha^*\otimes \beta')\alpha.$$ Arguing as in the proof of Lemma \ref{holomorphic}, we may take the inner product with $\beta'$ and integrate by parts to get: $$\int_X |\tr(\alpha^*\otimes \beta')|^2 + |\delbar_{AB}^*\beta|^2 = 0,$$ and we conclude that $\beta' = 0$ and $\delbar a = 0$ as claimed. \end{proof}

    %Next, we demonstrate an isomorphism of the first cohomology groups. By our lemma, $H^1_{\mathbb{C}, A, \alpha}$ is the quotient of $(a, \alpha)$ satisfying $\delbar a = \delbar \alpha'+a\alpha = 0$ by elements of the form $(-\delbar h, h\alpha).$

%If $\mathbf p$ has Property (C) and $\mathcal{M}(X, \mathbf p)$ is irreducible and Zariski smooth, then for each holomorphic triple $(A, \alpha, 0)$ in a fixed connected component of $\mathcal{M}(X, \mathbf p)$ the space $H^2_{\mathbb C, A, \alpha}$ has a constant dimension, and as before these spaces glue together to form a vector bundle over each component of $\mathcal{M}$, which we denote $\mathfrak{O}_\mathbb C.$ In this case, it is a complex vector bundle.

\subsection{Computing $e(\mathfrak{O})$}

In this section, we assume that $\mathbf p = (g, B, r\omega)$ has been chosen such that $F_B^{2,0} = 0$, with Property (C), and such that $\mathcal{M}(X, \mathbf p)$ is irreducible and Zariski smooth. Under this circumstance, we will provide a formula for the Euler class of the obstruction bundle $e(\mathfrak{O})$. 

Consider the space $\widetilde{\mathcal{M}}(X, \mathbf p)$, which we define to be the pre-image of $\mathcal{M}(X, \mathbf p)$ under the map $\mathscr{C}\to\mathscr{B}.$ Over this space, we may form the complex of trivial bundles $\mathcal{C}_\mathbb{C}\times \widetilde{\mathcal{M}}(X, \mathbf p).$ Since the complex $\mathcal{C}_\mathbb{C}$ is preserved by \emph{unitary} gauge transformations, we take the quotient of this trivial bundle complex by $\mathcal{G}$ to define a complex of vector bundles over $\mathcal{M}(X, \mathbf p)$ which we will also denote $\mathcal{C}_\mathbb{C}.$

In the situation when $\mathcal{M}(X, \mathbf{p})$ is Zariski smooth and irreducible, the first and second cohmologies of this complex $\mathcal{C}_\mathbb{C}$ are vector bundles over $\mathcal{M}(X, \mathbf{p}).$ The argument of \cite[Theorem 4.7]{Doan} shows that for $j=1,2,$ the isomorphisms $H^j_{A, (\alpha, 0)}\cong H^j_{\mathbb{C}, A, \alpha}$ vary continuously as we vary $(A, \alpha).$ This, combined $\mathcal{G}$-equivariance of the isomorphisms in Theorem \ref{deformation-theories} gives us:

\begin{theorem}\label{obstruction-bundle}
    Suppose $\mathbf p = (g, B, 2r\omega)$ has that $F_B^{2,0} = 0$ and has Property (C). If moreover $\mathcal{M}(X, \mathbf p)$ is irreducible and Zariski smooth, then for $j=1,2,$ the isomorphisms $H^j_{A, (\alpha, 0)}\cong H^j_{\mathbb{C}, A, \alpha}$ for each $(A, \alpha, 0)$ glue together to give a isomorphisms of real vector bundles. In the $j=2$ case we have in particular: $\mathfrak{O}\cong H^2(\mathcal{C}_\mathbb{C})$ over $\mathcal{M}(X, \mathbf p).$
\end{theorem}

Let $v$ denote the virtual complex dimension of $\mathcal{M}(X, \mathbf p)$ and $a$ the actual complex dimension, so that Theorem \ref{obstruction-bundle} guarantees that the bundle $\mathfrak{O}$ is in fact a complex vector bundle of rank $a -v,$ and moreover: $$e(\mathfrak{O}) = c_{a-v}(\mathfrak{O}).$$

The symbol complex of $\mathcal{C}_\mathbb{C}$ is the complex $\mathcal{S}_\mathbb C$ given as the sum of the complex: $$\Omega^0_X\xrightarrow{\delbar} \Omega^{0,1}_X\oplus\xrightarrow{\delbar} \Omega^{0,2}_X,$$ with the following complex shifted by 1 degree to the right: $$\Omega^0(E\otimes L)\xrightarrow{\delbar_{BA}}\Omega^{0,1}(E\otimes L)\xrightarrow{\delbar_{BA}}\Omega^{0,2}(E\otimes L)$$ which, in exactly the same may be thought of as a complex of vector bundles over $\mathcal{M}(X, \mathbf p)$. Clearly $S_\mathbb C$ is an elliptic complex, so its index is well-defined as an element of $K(\mathcal{M}(X, \mathbf p)).$ By the homotopy invariance of index in families, and the fact that the 0th cohomologies of $\mathcal{S}_\mathbb C$ and $\mathcal{C}_\mathbb C$ are clearly trivial as bundles over $\mathcal{M}$, the index $\ind S_\mathbb C$ is stably equivalent to $[H^2(\mathcal{C}_\mathbb C)] - [H^1(\mathcal{C}_\mathbb C)]$ in the $K$-theory of $\mathcal{M}.$ But $H^1(\mathcal{C}_\mathbb C)$ is the tangent bundle $T\mathcal{M},$ and $H^2(\mathcal{C}_\mathbb C)$ is precisely the bundle $\mathfrak{O},$ so we have the formula: \begin{equation}\label{e(o)} e(\mathfrak{O}) = c_{a-v}(\mathfrak{O}) = \left[c(T\mathcal{M})^{-1}c(\ind \mathcal{S}_\mathbb{C})\right]_{a-v}.\end{equation}

\section{Compactness Theorems}

This section contains the main novel observation of this paper, which is that a very simple condition on the parameter $B$ implies that the multi-spinor moduli space is compact for all sufficiently nearby parameters.

\subsection{Weakly Unobstructed ASD Connections}

We see from the holomorphic description of the moduli spaces that different choices of $B$ that induce the same holomorphic structure $\mathscr{E}$ on $E$ result in the isomorphic moduli spaces. In the case when $\mathscr{E}$ is stable, we therefore lose little by assuming that $B$ is in fact an anti-self-dual Yang-Mills connection, i.e. $\Lambda F_B = 0$ as well as $F_B^{2,0} = 0.$ Recall that to each anti-self-dual Yang-Mills connection $B$ on $E$ we may associate a deformation complex: $$\Omega^0_X(\mathfrak{su} E) \xrightarrow{d_B}\Omega^1_X(\mathfrak{su} E)\xrightarrow{d_B^+}\Omega^+_X(\mathfrak{su} E),$$ with cohomologies $H^j_B.$ When $X$ is K\"ahler, anti-self-dual Yang-Mills connections are holomorphic; if moreover $E$ is irreducible, then for each $j=0,1,2,$ the cohomology group $H^j_B$ agrees with the sheaf cohomology group $H^j_X(\End_0 E_B)$ (see \cite[Section 6.4.2]{DK}.)

\begin{definition}
    Let $X$ be a Riemannian 4-manifold and $E\to X$ a complex vector bundle with structure group $SU(N)$. We say a $SU(N)$ anti-self-dual Yang-Mills connection $B$ is \emph{unobstructed} if $H^2_B = 0.$

    If, moreover, $X$ is K\"ahler, we say a $SU(N)$ anti-self-dual Yang-Mills connection $B$ is \emph{weakly unobstructed} if every element in $H^0_X((\End_0 E_B)\otimes K_X)$ with rank $\leq 1$ at each point is identically zero. We abbreviate weakly unobstructed anti-self-dual Yang-Mills as WUASD.
\end{definition}

\begin{remark}
    By Serre duality, and the fact that the bundle $\End_0 E_B$ is isomorphic to its own dual, $H^0_X((\End_0 E_B)\otimes K_X)$ is dual to the obstruction space $H^2_X(\End_0 E_B)$. From \cite[Section 6.4.2]{DK} we have an isomorphism $H^2_B\cong H^2_X(\End_0 E_B)\oplus H^0_B\cdot\omega,$ hence unobstructedness indeed implies weak unobstructedness.
\end{remark}

A surprisingly simple argument shows that $\mathcal{M}(X, g, B, 2r\omega)$ is compact whenever $B$ is WUASD:

\begin{theorem}
    If $B$ is WUASD with respect to the K\"ahler metric $g$, then for all sufficiently large $r,$ the parameter $(g, B, 2r\omega)$ has Property (C), so in particular $\mathcal{M}(X, g, B, 2r\omega)$ is compact.
\end{theorem}

\begin{proof}
    Let $B$ be as in the statement of the theorem and $(\mathscr{L}, \alpha, \beta)$ be a $B$-holomorphic triple. Then, $\alpha\otimes\beta^*\in H^0_X(\End \mathscr{E}\otimes K_X)$ is traceless, so it defines a section of $H^0_X(\End_0 \mathscr{E}\otimes K_X),$ By the WUASD assumption, all elements of this space with rank $\leq 1$ pointwise are identically zero. Hence, $\alpha\otimes\beta^* = 0,$ so at each point $x\in X,$ one of $\alpha(x)$ or $\beta(x)$ must equal 0. Since they are both holomorphic, one must vanish identically by unique continuation. The integral of the third holomorphic Seiberg-Witten equation then gives us: $$\pi\left(2c_1(L)-c_1(K_X)\right)\cdot[\omega] +r[\omega]^2 = ||\alpha||_{L^2}^2 - ||\beta||_{L^2}^2$$ by Chern-Weil theory, so if we assume $r$ is sufficiently large the left-hand side is positive, so we must have $\beta = 0,$ which proves that $(g, B, 2r\omega)$ has Property (C), and compactness follows from Lemma \ref{property-c}.
\end{proof}

In fact, under this same assumption, we will prove a much stronger compactness theorem, which is the main theorem of this section:

\begin{theorem}\label{strong_compact}
    If $B$ is WUASD, then there exists some open neighborhood $\mathcal{U}$ of $\mathbf p = (g, B, 2r\omega)$ in the space $\text{Met}(M)\times \mathcal{A}(E)\times \Omega^2(M)$ such that for any $(g', B', \eta') \in \mathcal{U}$ the moduli space $\mathcal{M}(X, g', B', \eta')$ is compact. In other words, $\mathbf p \in \mathscr{U}_{cpt}.$
\end{theorem}

\begin{remark}
    This can be viewed as a direct analog of \cite[Lemma 2.25]{Doan} in the Riemann surface case.
\end{remark}

The main ideas of the proof of this theorem, which will be carried out in the next two sections, are a general compactification for Seiberg-Witten-like moduli spaces due to Taubes \cite{Taubes16} and a holomorphic interpretation of the limiting configurations in this compactification, which may be of independent interest.

\subsection{Review of Taubes Compactification}

We will first review Taubes' compactification of the multi-spinor moduli space. Here, the setting is any Riemannian manifold $(X, g)$, not necessarily K\"ahler, and with any $SU(N)$-connection $B.$

\begin{remark}
    I will follow \cite{Doan} and use the term \emph{limiting configuration}, not found in Taubes' original papers.
\end{remark}

In the below, let $R_{spin}$ denote the curvature of the spin$^c$ connection $\nabla$ on $\mathbb{S}^+$ defined by the metric $g$.

\begin{definition}\label{limiting-config}
    A limiting configuration $(A', \Phi, Z)$ on $X$ consists of the data:\begin{enumerate}
        \item $Z\subset X$ a closed, nowhere dense subset.
        \item A unitary connection $A'$ on the line bundle $S|_{X - Z}$.
        \item A section $\Phi\in \Gamma_{X-Z}(W^+\otimes E)$ satisfying: \begin{enumerate}
            \item $||\Phi||_{L^2(X-Z)} = 1.$
            \item The norm $|\Phi|$ extends to a H\"older continuous function on $X$ which is zero on $Z.$
            \item $|\nabla_{AB}\Phi|\in L^2(X-Z).$
            \item $D_{A'\otimes B}^+\Phi = 0$
            \item $\mu(\Phi, \Phi) = 0$
            \item $F_{A'} = - |\Phi|^{-2}\langle\Phi, (R_{spin}+F_B)\Phi\rangle$
        \end{enumerate}
    \end{enumerate}
\end{definition}

\begin{definition} \label{convergence}
    A sequence $(A_i',\Phi_i)$ is said to \emph{converge} to a limiting configuration $(A', \Phi, Z)$ if: \begin{enumerate}
        \item On compact subsets of $X-Z,$ \begin{enumerate}
            \item $A_i'\to A'$ in the weak $L^2_1$ topology, and
            \item $||\Phi_i||_{L^2}^{-1}\Phi_i\to \Phi$ in the $L^2_2$-topology.
        \end{enumerate}
        \item On $X,$ \begin{enumerate}
            \item $||\Phi_i||_{L^2}^{-1}|\Phi_i|\to |\Phi|$ in $L^2_1$ and $C^0$, and
            \item $||\Phi_i||_{L^2}^{-1}|\nabla_{A_iB}\Phi_i|\to |\nabla_{AB}\Phi|$ in $L^2$.
        \end{enumerate}
    \end{enumerate}
\end{definition}

With these definitions at hand, we may state Taubes' theorem:

\begin{theorem}[{\cite{Taubes16}}]\label{taubes}
    Let $(A_i',\Phi_i, g_i, B_i, \eta_i)$ be a sequence of solutions to Equation \ref{msw}, with $(g_i, B_i, \eta_i)\to (g, B, \eta)$ in $C^\infty.$ Then, \begin{enumerate}
        \item If $\liminf\limits_{i\to\infty}||\Phi_i||_{L^2}<\infty,$ there is a subsequence of $(A_i',\Phi_i)$ which converges modulo gauge to a solution $(A', \Phi)$ of Equation \ref{msw} with respect to $(g, B).$
        \item If $\liminf\limits_{i\to\infty}||\Phi_i||_{L^2}=\infty,$ there is a subsequence of $(A_i',\Phi_i)$ which converges modulo gauge to a limiting configuration $(A', \Phi, Z).$
    \end{enumerate}
\end{theorem}

\begin{remark}
    In Taubes' paper \cite{Taubes16}, the above theorem is stated in the weaker form where $(g_i, B_i, \eta_i) = (g, B, \eta)$ for all $i.$ The stronger form that we need here follows immediately from the proofs in that paper \cite{Taubes-talk}.
\end{remark}

We will need the following lemma in Section 5:

\begin{lemma}\label{fabounded}
    If $(A', \Phi, Z)$ is a limiting configuration on the compact 4-manifold $X,$ then $||F_{A'}||_{C^0(X-Z)}$ is finite.
\end{lemma}

\begin{proof}
    Theorem \ref{taubes} and Condition (f) in Definition \ref{limiting-config} tell us the pointwise $C^0$-bound $$|F_{A'}|\leq |R_{spin}+F_B|_{op},$$ where $|\cdot|_{op}$ denotes the operator norm of a matrix. This quantity is continuous on the whole manifold $X$, so it is bounded.
\end{proof}

\subsection{Proof of Strong Compactness Theorem}

In this subsection, we will prove restrictions on limiting configurations on K\"ahler surfaces and then use these facts to finish the proof of Theorem \ref{strong_compact}. So, for the remainder of this section, let $(A', \Phi, Z)$ be a fixed limiting configuration for $g$ a K\"ahler metric and $B$ a connection with $F_B^{2,0} = 0.$ As in the previous section, we write $\Phi = (\alpha, \beta)$ under the identification $\mathbb{S}^+\cong \underline{\mathbb{C}}\oplus K_X,$ and use $A$ such that $A^2\otimes \nabla_0' = A'.$

Under this identification of the positive spinor bundle $\mathbb{S}^+$ with $\underline{\mathbb{C}}\oplus K_X,$ its canonical metric spin$^c$ connection as the sum of the product connection with the Chern connection of $K_X$. Hence, as a 2-form-valued matrix, the curvature tensor $R_{spin}$ in the previous section is given as: $$R_{spin} = \begin{pmatrix} 0 & 0 \\ 0 & \rho \\ \end{pmatrix},$$ where $\rho(u, v) = \text{Ric}(iu, v)$ is the Ricci form.

First, we show that $A, \alpha,$ and $\beta^*$ are holomorphic away from $Z$:

\begin{lemma}
    The connection $A$ satisfies $F_A^{2,0} = 0,$ and on $X-Z$ we have the pointwise equations: $$\delbar_{AB}\alpha = 0, ~\delbar^*_{AB}\beta = 0.$$
\end{lemma}

\begin{proof}
First, we show that $F_A^{2,0} = 0.$ On $X - Z,$ the connection $A$ has curvature satisfying: $$F_A = - |\Phi|^{-2}\langle\Phi, (R_{spin}+F_B)\Phi\rangle.$$ By the description of $R_{spin}$ in terms of the Ricci form, clearly $R_{spin} = R_{spin}^{1,1},$ so since $F_B^{2,0} = 0$ by assumption, we must have $F_A^{2,0} = 0$, so the $(0,1)$ part of $\nabla_A$ defines a holomorphic structure on $L|_{X-Z}.$

Next, imitating the proof strategies of \cite[Proposition 6.4]{Doan} and \cite[Theorem 3.2]{Bryan-Wentworth}, will we show that $\alpha$ and $\beta^*$ are holomorphic on $X - Z.$ Recall, we have the equations on $X-Z$: \begin{equation}\label{x-z} \delbar_{AB}\alpha + \delbar_{AB}^*\beta = 0, \tr(\alpha\otimes \beta^*) = 0, |\alpha|^2 = |\beta|^2\end{equation}

Let $f:\mathbb{R}\to[0,1]$ be a smooth, increasing function such that $f(x) = 0$ when $x\leq 0$ and $f(x) = 1$ when $x\geq 1.$ For $\epsilon\in(0, 1),$ let $\rho_\epsilon:X\to [0,1]$ be the function: $$\rho_\epsilon(x) = f\left(\frac{|\Phi(x)| - \epsilon}\epsilon\right),$$ so that $\rho_\epsilon(x) = 0$ whenever $|\Phi(x)|\leq \epsilon$ and $\rho_\epsilon(x) = 1$ whenever $|\Phi(x)|\geq 2\epsilon,$ and moreover $\rho_\epsilon$ is differentiable. Applying $\delbar_{AB}$ to the first equation of \ref{x-z} gives us: $$\delbar_{AB}\delbar_{AB}\alpha + \delbar_{AB}\delbar_{AB}^*\beta = 0,$$ which, by $F_A^{2,0} = F_B^{2,0} = 0$ implies $\delbar_{AB}\delbar_{AB}^*\beta = 0.$ As a consequence of this and the pointwise identity: $$\delbar^*_{AB}(\rho_\epsilon\beta) = \{\delbar\rho_\epsilon, \beta\} + \rho_\epsilon\delbar^*_{AB}\beta,$$ where $\{\cdot, \cdot\}$ is some bilinear form with bounded norm, we may integrate by parts: $$0 = \int_X\langle \delbar_{AB}\delbar^*_{AB}\beta, \rho_\epsilon \beta\rangle = \int_X \{\delbar\rho_\epsilon, \beta, \delbar_{AB}^*\beta\} + \int_X \rho_\epsilon{\left|\delbar^*_{AB}\beta\right|}^2,$$ where $\{\cdot, \cdot, \cdot\}$ is some trilinear form with bounded norm. Our task is to show the rightmost term approaches 0 as $\epsilon\to 0$; we do this as follows: $$\int_X \rho_\epsilon{\left|\delbar^*_{AB}\beta\right|}^2\leq\left|\int_X \{\delbar\rho_\epsilon, \beta, \delbar_{AB}^*\beta\}\right|\leq C_1\int_X |\nabla\rho_\epsilon|\cdot|\beta|\cdot|\nabla_{AB}\beta|.$$ Using the following pointwise estimate from Kato's inequality, $$|\nabla\rho_\epsilon|=\epsilon^{-1} \left|f'\left(\frac{|\Phi| - \epsilon}\epsilon\right)\right| |\nabla|\Phi||\leq \epsilon^{-1}||f||_{C^1} |\nabla_{AB}\Phi|,$$ as well as the fact that $\nabla\rho_\epsilon$ is supported on the set $|\Phi(x)|\leq 2\epsilon,$ we have that: $$\int_X |\nabla\rho_\epsilon|\cdot|\beta|\cdot|\nabla_{AB}\beta|\leq C_2\int_{|\Phi(x)|\leq2\epsilon} \epsilon^{-1}2\epsilon |\nabla_{AB}\Phi|^2.$$ Since $|\nabla_{AB}\Phi|^2$ is integrable on $X,$ the dominated convergence theorem tells us that the integral $\int_{|\Phi(x)|\leq2\epsilon} |\nabla_{AB}\Phi|^2$ approaches 0 as $\epsilon\to 0.$ Hence, $$\lim_{\epsilon\to 0}\int_{|\Phi(x)|\geq 2\epsilon}{\left|\delbar^*_{AB}\beta\right|}^2\leq\lim_{\epsilon\to 0}\int_X\rho_\epsilon{\left|\delbar^*_{AB}\beta\right|}^2 = 0,$$ so $\delbar^*_{AB}\beta(x) = 0$ for each $x\in X - Z$. Thus the first equation of \ref{x-z} implies $\delbar_{AB}\alpha(x) = 0$ for each $x\in X - Z$ as well.
\end{proof}

\begin{lemma}\label{Zanalytic}
    If $(A, (\alpha, \beta), Z)$ is a limiting configuration, then $\alpha\otimes\beta^*$ is holomorphic on all of $X$ and $Z$ is a closed subset of an analytic subvariety. In particular, $Z$ has real Hausdorff dimension at most $2.$
\end{lemma}

\begin{proof}
    $\alpha\otimes \beta^*$ is a holomorphic section of $\End_0\mathscr{E}\otimes K_X$ on $X-Z,$ which limits to 0 on $Z.$ Hence, $\alpha\otimes \beta^*$ extends by 0 on $Z$ to a continuous section of $\End_0\mathscr{E}\otimes K_X$ defined over all of $X$, which is bounded and holomorphic outside $Z$. 

    In fact, $\alpha\otimes \beta^*$ is holomorphic on all of $X$. One way to see this is as follows. Since $\alpha\otimes \beta^*$ is holomorphic outside its zero set, it is clear that the function $u:=\log |\alpha\otimes \beta^*|:X\to\mathbb{R}$ is plurisubharmonic. Since $Z$ is nowhere dense, the function $u$ is not identically equal to $-\infty$ in any open neighborhood of $X,$ hence $Z=u^{-1}(-\infty)$ is a pluripolar set. (This already shows that $Z$ has real Hausdorff codimension $\geq 2,$ since this is a property of pluripolar sets.)
    
    Since $\alpha\otimes\beta^*$ is holomorphic and locally bounded in a neighborhood of a pluripolar set, it extends to a holomorphic function on all of $X$ (see \cite[Corollary 5.25, Chapter 1]{Demailly}.) Moreover, it must extend by 0 over $Z$ by continuity, hence $Z$ is a closed subset of the analytic set $(\alpha\otimes\beta^*)^{-1}(0).$ 
\end{proof}

\begin{remark}
    Tanaka finds in \cite[Theorem 1.2]{Tanaka} that when the rank of $E$ is 2, $X$ is K\"ahler and $B$ is holomorphic, then the set $Z$ is complex analytic. In \cite[Proposition 1.2]{Taubes16}, Taubes showed that when the rank of $E$ is 2, but under no assumptions on the metric or on $B,$ the set $Z$ has real Hausdorff dimension at most 2. Our Lemma \ref{Zanalytic} can therefore be viewed as a generalization of these results to arbitrary rank, provided we assume $X$ is K\"ahler and $B$ is holomorphic.
\end{remark}

\begin{lemma}\label{unobstructed}
    If $B$ is WUASD then there are no limiting configurations with respect to $(X, g, B).$
\end{lemma}

\begin{proof}
    If there is a limiting configuration $(A, (\alpha, \beta), Z)$, then $\alpha\otimes \beta^*$ is a holomorphic section of $\End_0\mathscr{E}\otimes K_X$ on $X$ by Lemma \ref{Zanalytic}. Since $\alpha\otimes \beta^*$ is clearly rank $\leq 1$ at each point, it must be identically zero since $B$ is WUASD. Hence, at each point $x,$ either $\alpha(x) = 0$ or $\beta(x) = 0.$ Since $\alpha$ and $\beta^*$ are both holomorphic, the unique continuation theorem requires either $\alpha = 0$ or $\beta = 0.$ 
    
    Letting $\gamma:\Omega^+_X\otimes \mathbb{C}\to \mathfrak{sl}(\mathbb{S}^+\otimes L)$ be the Clifford multiplication, we may compute: $$\Lambda\gamma^{-1}(\mu(\Phi, \Phi)) = \frac {\sqrt{-1}}4 \left(|\alpha|^2 - |\beta|^2\right).$$ So, if one of $\alpha$ or $\beta$ is 0, $\mu(\Phi, \Phi) = 0$ forces the other to also be 0. Hence, $\Phi = 0,$ which contradicts $||\Phi||_{L^2(X - Z)} = 1.$
\end{proof}

We may now complete the proof of the main compactness theorem: 

\begin{proof}[Proof of Theorem \ref{strong_compact}]
Suppose by contradiction that there is no such neighborhood $\mathcal{U}.$ Then, there is a sequence $(g_i, B_i, \eta_i)$ converging to $(g, B, \eta)$ such that $\mathcal{M}(X, g_i, B_i, \eta_i)$ is non-compact. By Taubes' Theorem \ref{taubes}, for each $k$ there must exist a solution $(A_k, \Phi_k, g_k, B_k, \eta_k)$ to Equation \ref{msw} with $||\Phi_k||_{L^2} > k.$ But, this now implies there exists a limiting configuration with respect to $(X, g, B),$ which contradicts Lemma \ref{unobstructed}.
\end{proof}

\section{Limiting configurations on K\"ahler surfaces}

We have already seen how rescaled limits of divergent sequences of multi-monopoles converge to limiting configurations. This section will be devoted to strengthening this notion of convergence when the manifold is K\"ahler and relevant parameters holomorphic. In particular, we will show that the convergence can be upgraded to $C^\infty_{\text{loc}}(X - Z)$ and that the curvatures converge to a current which contains the data of the codimension 1 part of the singular set. In the final subsection, we will determine an even stronger conclusion when the rank of $E$ is 2.

In the entirety of this section, let $X$ be a K\"ahler surface with metric $g,$ K\"ahler form $\omega,$ and let $B$ be a connection on $E$ such that $F_B^{2,0} = 0.$ 

%The goal of this section will be a dimension 2 analog of \cite[Theorem 1.5]{Doan}. Namely, when $B$ is not WUASD, the moduli space $\mathcal{M}(X, \mathbf{p})$ is not necessarily compact, and there are two compactifications: one from K\"ahler geometry and one from gauge theory. We will show that in fact these compactifications are homeomorphic to each other.

Most of our arguments are directly analogous to those in \cite{Doan} and \cite{DoanKW}, so we will highlight notable deviations from these.

\subsection{Holonomy Data of Limiting Configuration}

Here, we will show how to recover the first Chern class of the line bundle $L$ from the data of a limiting configuration. Let $(A', \Phi, Z)$ be a limiting configuration for a K\"ahler manifold and $F_B^{2,0}=0,$ so that $Z$ is analytic by Lemma \ref{Zanalytic}.

Let $Z_1,\dots, Z_k$ denote the irreducible codimension 1 components of $Z.$ To each $Z_j$, we associate a holomorphic line bundle $L_j\to X$ with a section $s_j\in H^0_X(L_j)$ that vanishes at $Z_j$; the section is unique up to multiplication by a nonzero complex number. We choose a hermitian metric on each $L_j.$

For each $Z_j,$ we will associate an integer $q_j$ as follows. First, choose a small disk $D\subset X$ which intersects $Z_j$ transversely at a single smooth point, and intersects no other point of $Z$. Orient $D$ so that the intersection is positive. Choose a unitary trivialization $\tau: S|_{D}\to D\times \mathbb{C},$ so that on $D$ we may write $A' = d + a$ for some purely imaginary 1-form $a\in \Omega^1_{D - \{0\}}.$ Thinking of $D$ as the image of a smooth map from the unit disk $\mathbb{D}\subset \mathbb{C}$ to $X,$ we define $D_r$, for each $r\in(0,1],$ to be the image of the disk of radius $r$ centered at the origin in $\mathbb{C}.$ We will define $q_j$ as a limit, which a priori depends on the above data.

\begin{lemma}
    The following limit exists and is moreover an integer: $$q_j(D, \tau) := \lim_{r\to 0^+}\frac{\sqrt{-1}}{2\pi}\int_{\partial D_r} a.$$
\end{lemma}

\begin{proof}
    %We use an argument based on that of \cite[Theorem 1]{Shevchishin}. Since $F_{A'}$ is pointwise bounded, it extends to a current $\widetilde{K}$ on all of $D,$ which is closed by the Harvey-Polking Theorem \cite[Theorem 2.5]{Harvey-Polking}. 
    
    Since the de Rham cohomology $H^2_{dR}(D-\{0\}; \mathbb{R}) = 0,$ and by the Bianchi identity $F_{A'}$ is closed, there must be a 1-form $\sigma$ on $D - \{0\}$ such that $F_{A'} = d\sigma$.     
    
    Therefore, $A' - \sigma$ is flat. So, there exists some unitary gauge transformation $u:D-\{0\}\to S^1$ such that $u(A' - \sigma)$ is the product connection on the trivial bundle. Hence, for each $r\in(0, 1]$: $$0 = \int_{\partial D_r} u(a - \sigma) = \int_{\partial D_r} a -\int_{\partial D_r}\sigma - \int_{\partial D_r}u^{-1}du.$$ Restricting $u$ to $\partial D_r,$ we get a map $u:\partial D_r \to S^1,$ and we have the following pullback formula: $$u^*\left(\frac1{2\pi} d\theta\right) = -\frac{\sqrt{-1}}{2\pi}u^{-1}du.$$ Since $\int_{S^1}d\theta = 2\pi,$ we have: \begin{equation*}
        \begin{split}
            \frac{\sqrt{-1}}{2\pi}\int_{\partial D_r}a-\frac{\sqrt{-1}}{2\pi}\int_{\partial D_r}\sigma &= \frac{\sqrt{-1}}{2\pi}\int_{\partial D_r}u^{-1}du \\
            &= -\int_{\partial D_r} u^*\left(\frac1{2\pi} d\theta\right) \\ &= -\frac1{2\pi}(\deg u)\int_{S^1} d\theta \\ &= -\deg u.
        \end{split}
    \end{equation*} The $C^0(X-Z)$ bound from Lemma \ref{fabounded} now implies $$\lim_{r\to0^+}\int_{\partial D_r}\sigma = \lim_{r\to0^+}\int_{D_r}d\sigma = \lim_{r\to0^+}\int_{D_r}F_{A'} = 0,$$ so we get $q_j(D, \tau) = -\deg u \in \mathbb{Z},$ which proves that the limit exists and is an integer, as desired.
\end{proof} 

\begin{lemma}\label{q}
    The integer $q_j(D, \tau)$ is independent of the choices of $\tau$ and $D.$
\end{lemma}

\begin{proof}
    For a different choice of trivialization $\tau' = e^{\sqrt{-1}f}\tau,$ the connection matrix of $\xi$ is now given by $a-\sqrt{-1}df$, so by Stokes' theorem: $$q_j(D, \tau') = \lim_{r\to 0^+}\frac{\sqrt{-1}}{2\pi}\int_{\partial D_r} a-\sqrt{-1}df = \lim_{r\to 0^+}\frac{\sqrt{-1}}{2\pi}\int_{\partial D_r} a=q_j(D, \tau)$$ Let $D'$ be a different such disk. Then, since the subset of $Z_j$ consisting of smooth points that intersect no other component of $Z$ is connected, for each $r\in(0, 1]$ there exists a cobordism $W_r$ from $D_r$ to $D'_r$ such that $\partial W_r$ intersects $Z$ only in the intersections of $D_r\cap Z_j$ and $D'_r\cap Z_j.$ There is an annulus $V_r\subset \partial W_r$ which is itself a cobordism from $\partial D_r$ to $\partial D'_r,$ and $V_r\cap Z = \emptyset.$
    
    Choose a unitary trivialization $\tau$ of $M$ on $W_1$, so that $A' = d + a$ on $W_1.$ By Stokes' theorem: $$\int_{\partial D'_r} a -\int_{\partial D_r} a = \int_{\partial V_r}a = \int_{V_r} da = \int_{V_r} F_{A'}.$$ 
    
    Hence, Lemma \ref{fabounded} tells us: $$\lim_{r\to0^+}\left|\int_{V_r} F_{A'}\right|\leq \lim_{r\to0^+} \text{area}(V_r)\cdot||F_{A'}||_{C^0(X-Z)} = 0,$$ and we conclude that $q_j(D, \tau) = q_j(D', \tau).$
\end{proof}

By Lemma \ref{q}, we may define $q_j$ to be the integer $q_j(D, \tau)$ for any choice of $D, \tau.$ \begin{comment}We will need the following higher-dimensional analog of \cite[Lemma 5.19]{Doan}:

\begin{lemma}\label{519}
    Let $D$ be as above, and suppose that on some open neighborhood $U\supset D$ we have a holomorphic section $\varphi\in H^0(U - Z; M_\xi)$ satisfying $|\varphi| = 1.$ Then, $|\varphi|_{\partial D}:\partial D\to S^1$ defines a map of the circle to itself whose degree equals $q_j.$
\end{lemma}

\begin{proof}
    By homotoping $D$ if necessary, it may be assumed to be the image of a holomorphic embedding of the unit disk $\mathbb{D}$, and by pulling back $M_\xi$ and $\varphi$ via the holomorphic embeddings $\mathbb{D}\to U$ and $\mathbb{D} -\{0\} \to U - Z$ we reduce this lemma to the case of \cite[Lemma 5.19]{Doan}.
\end{proof}
\end{comment}

\subsection{$C^\infty_{\text{loc}}$ convergence}

In this part, we will show that for a fixed K\"ahler parameter $(g, B, 2r\omega),$ a sequence of multi-monopoles converging to a limiting configuration in fact converge in a stronger topology than that given by Theorem \ref{taubes}. The below theorem is a direct adaptation of Doan's \cite[Theorem 2.2]{DoanKW}, an analogous theorem taking place on a Riemann surface, to our setting:

\begin{remark}
    In this subsection and the next, we will occasionally use the letter $\gamma$ to denote $\beta^*$ for notational convenience. We will also rewrite the equations so that the section $\Phi$ always has $||\Phi||_{L^2} = 1$ and to correct for the rescaling introduce a new parameter $\epsilon,$ defined below:
\end{remark}

\begin{theorem}\label{c-infty}
    Let $(X, g, \omega)$ be a K\"ahler surface, $E\to X$ a $SU(N)$-bundle with compatible connection $B$ satisfying $F_B^{2,0} = 0,$ and $L\to X$ a hermitian line bundle. Suppose that for each $j\in\mathbb{N},$ we have a unitary connection $A_j\in\mathscr{A}(L)$, a section $\alpha_j\in \Gamma_X(E\otimes L)$ a section $\gamma_j\in \Gamma_X(E^*\otimes L^*\otimes K_X)$, and a number $\epsilon_j\in(0,\infty)$ satisfying: \begin{equation}\label{rescaledholomsw}
        \begin{cases}
            \delbar_{A_jB}\alpha_j = 0 \\
            \delbar_{A_jB}\gamma_j = 0 \\
            F_{A_j}^{2,0} = 0 \\
            \epsilon_j\left(-\sqrt{-1}\Lambda F_{A_j} +\frac12\sqrt{-1}\Lambda F_{\nabla_0}+r \right) = \frac14\left(|\alpha_j|^2-|\gamma_j|^2\right) \\
            ||\alpha_j||_{L^2}^2+||\gamma_j||_{L^2}^2 = 1 \\
            \lim_{j\to\infty}\epsilon_j = 0 \\
        \end{cases}
    \end{equation}

    Then, there exists a limiting configuration $(A, Z, (\alpha, \gamma^*))$ and unitary gauge transformations $u_j\in\mathcal{G}$ such that a subsequence of $(u_j\cdot A_j, u_j\alpha_j, u_j^{-1}\gamma_j^*)$ converges to $(A, \alpha, \gamma^*)$ in the $C^\infty$ topology on compact subsets of $X - Z.$
\end{theorem}

\begin{remark}
    In beautiful recent work, He, Parker, and Walpuski showed a deep bootstrapping result which implies Theorem \ref{c-infty} in the 2-spinor case without invoking the K\"ahler assumption \cite[Theorem 1.7]{HPW}.
\end{remark}

\begin{proof}
    Throughout this proof, we will continue to denote all subsequences taken with the same lettering to avoid multiple lower indices.
    
    By Theorem \ref{taubes}, there exist unitary gauge transformations $u_j\in\mathcal{G}$ such that a subsequence of $(u_j\cdot A_j, u_j\alpha_j, u_j^{-1}\gamma_j^*)$ converges to a limiting configuration $(A_0, \alpha_0, \gamma^*_0)$ in the weak $L^2_1$ topology on compact subsets of $X - Z_0.$

    By \cite[Step (1), proof of Theorem 2.2]{DoanKW}, there exist smooth real-valued functions $f_j:X\to \mathbb{R}$ and unitary gauge transformations $v_j$ such that a subsequence of $$(\widetilde{A_j}, \widetilde{\alpha_j}, \widetilde{\gamma_j^{*}}):= e^{f_j/2}v_j\cdot(A_j, \alpha_j, \gamma_j^*)$$ converges on all of $X$ to some triple \emph{which is not necessarily a solution to the 4th or 5th equation in the statement of the theorem} $(\widetilde A, \widetilde \alpha, \widetilde{\gamma_j^{*}})$ in $C^\infty,$ where we view $e^{f_j}\in\mathcal{G}_\mathbb{C}.$ Moreover, $\alpha'$ and $\gamma_j'$ are both nonzero. Without loss of generality we assume all $v_j = 1.$ Let $Z$ denote the union of the zero sets of $\widetilde \alpha$ and $\widetilde{\gamma}$. Since $\widetilde \alpha$ and $\widetilde{\gamma}$ are holomorphic and nonzero, we must have that $Z$ is an analytic set which is nowhere dense.

    Steps (3) and (4) of Doan's proof of \cite[Theorem 2.2]{DoanKW} apply exactly to our setting, and they tell us that provided it is not true that: $$\liminf_{j\to\infty} ||\left( |\alpha_j|\cdot|\gamma_j|\right)||_{L^\infty(X)}=0,$$ then up to taking a subsequence, $f_j$ converges to some smooth function $f:X - Z\to\mathbb{R}$ in $C^\infty_{\text{loc}}(X - Z).$  But this assumption is indeed satisfied, since, on any compact set $K\subset X - Z_0$ with non-empty interior, $\alpha$ and $\gamma$ vanish nowhere. By $L^2_{\text{loc}}$ convergence we have that: $$0<\int_K|\alpha|^2\cdot|\gamma|^2 =\lim_{j\to \infty} \int_K|\alpha_j|^2\cdot|\gamma_j|^2 \leq \vol(K) \liminf_{j\to\infty} ||\left( |\alpha_j|\cdot|\gamma_j|\right)||_{L^\infty(K)}^2.$$ Hence, some subsequence of $(A, \alpha_j, \gamma_j^*)$ converges to $(A, \alpha, \gamma_j^*) := e^{-f/2}\cdot(\widetilde{A_j}, \widetilde{\alpha_j}, \widetilde{\gamma_j^{*}})$ in $C^\infty_{\text{loc}}(X - Z).$ We'd like to show that $(A, Z, (\alpha, \gamma^*))$ is a limiting configuration and that this convergence is also in the sense of Definition \ref{convergence}. All the properties of these definitions on $X - Z$ follow immediately from the $C^\infty$ convergence on compact sets.

    Note $|\alpha_j| = |u_j\alpha_j|,$ which approaches $|\alpha_0|$ in $C^0(X)$ as $j\to \infty.$ Hence, $|\alpha_0| = |\alpha|$ on $X - Z_0 \cup Z.$ By the nowhere density of $Z_0$ and $Z,$ we must have that $|\alpha|$ extends to a $C^0$ function on all of $X$ and $|\alpha_j|\to |\alpha|$ in $C^0(X)$ as $j\to \infty$. (Moreover this shows $Z = Z_0.$) The analogous statement for $|\gamma|$ and for $|\nabla_{A_j}\Phi_j|$ on all of $X$ follow in exactly the same way.
\end{proof}

\subsection{Convergence of Currents}

We will show another interpretation of the singular set of a limiting configuration, which is that the curvatures of an unbounded sequence of multi-monopoles converge to a current, and the complex codimension 1 part of the singular set is precisely the singular part of the current.

For the decomposition of $Z$ as above, each $Z_j$ defines a cohomology class $[Z_j]\in H^2(X; \mathbb{Z})$, which equals the Poincar\'e dual of $Z_j$ if $Z_j$ is smoothly embedded (see for instance \cite[p. 61]{GH}).

We first require an \emph{a priori} $L^1$ bound on solutions to the same modified Kazdan-Warner equation of Bryan and Wentworth that we first saw in Section 3.

\begin{definition}
    Let $(X, g)$ be a smooth $n$-dimensional manifold, with volume measure denoted by $\Omega$. Denote by $\omega_k$ the volume of the unit ball in $\mathbb{R}^k.$ For each integer $m\in[0,n],$ the $m$-dimensional \emph{upper Minkowski content} of a subset $Z\subset X$ is defined as: $$\mathcal{M}^{*,m}(Z) := \limsup_{r\to 0^+} \frac{\Omega(\left\{x\in X:\text{dist}(x, Z)\leq r\right\}}{\omega_{n-m}r^{n-m}}.$$
\end{definition}

The $m$-dimensional upper Minkowski content of a compact rectifiable set is equal to its $m$-dimensional Hausdorff measure (see \cite[Theorem 3.2.39]{Federer}.) The following lemma about upper Minkowski content follows immediately from this fact together with the observation that complex analytic subvarieties are rectifiable:

\begin{lemma}\label{minkowski}
    A non-trivial complex analytic subvariety $Z$ of a connected compact complex manifold $X$ has zero $(\dim_\mathbb{R} X - 1)$-dimensional upper Minkowski content.
\end{lemma}

\begin{lemma}\label{l1}
    Let $(X, g)$ be a compact Riemannian $n$-dimensional manifold (possibly with boundary), with H\"older functions $A, B\in C^{0, \gamma}(X)$ for some exponent $\gamma>0.$ Suppose also that the set $(A^{-1}(0)\cup B^{-1}(0))$ has $(n-1)$-dimensional upper Minkowski content equal to 0. Let $w:X\to\mathbb{R}$ be any smooth function on $X$, and let $\epsilon_0>0$. Then, there exists a constant $M$ depending only on $\epsilon_0, \gamma, \vol(X)$ and bounds for $||w||_{C^0(X)}, ||A||_{C^{0,\gamma}(X)},$ and $||B||_{C^{0,\gamma}(X)},$ with the following property:
    
    For any $\epsilon \in (0, \epsilon_0]$ and smooth function $f:X\to \mathbb{R}$ satisfying the equation: $$\epsilon\Delta f+ Ae^f-Be^{-f}+w =0,$$ we have the estimate: $$||f||_{L^1(X)}\leq M.$$%with holomorphic Hermitian vector bundles $E_1, E_2\to X$, and holomorphic sections $\alpha\in H^0(E_1, X), \beta\in H^0(E_2, X)$ both not identically zero.
\end{lemma}

\begin{proof}
    Let $\Omega$ be a bounded region with $\overline\Omega\subsetneq X,$ on which $\min(A, B) \geq \eta$. We will first show an $L^1$ bound on $e^f$ in this region.

    \begin{remark}
    Throughout the proof of this lemma, we will use the convention that the constant $C$ may vary from line to line, but depends only on the $C^0$ norms of the functions $A, B,$ and $w$ on all of $X,$ and the volume of $X.$ In particular, it is independent of $\eta$ and $\Omega.$
    \end{remark}

    Note that this equation implies $u:= e^f$ satisfies the differential inequality $$\epsilon\Delta u + Au^2 + (-B+wu) \leq 0.$$ Doan proves in \cite[Equation (3.5)]{DoanKW} that this differential inequality implies the bound $$||u||_{L^3(\Omega)}\leq C\min(1, \eta^{-1}),$$ where the constant $C$ depends only on $\epsilon_0,$ bounds on the $C^0(X)$ norms of $A, B,$ and $w,$ as well as the volume of $X.$ Using H\"older's inequality, we get an $L^1$ bound: $$||u||_{L^1(\Omega)}\leq C\min(1, \eta^{-1})\vol(\Omega)^{2/3}.$$ An analogous differential inequality and argument for $e^{-f}$ also holds, giving us: $$||e^{|f|}||_{L^1(\Omega)}\leq C\min(1, \eta^{-1})\vol(\Omega)^{2/3}.$$

    Since the function $t\mapsto e^t$ is convex, we may use Jensen's inequality to conclude that: $$\frac1{\vol(\Omega)}||f||_{L^1(\Omega)}\leq \log\left(\frac1{\vol(\Omega)}||e^{|f|}||_{L^1(\Omega)}\right)\leq C\min(0, -\log \eta-\frac13\log\vol(\Omega))$$

    For $x, y\in X,$ let $\rho(x, y)$ be the Riemannian distance between these points. Let $Z = A^{-1}(0)\cup B^{-1}(0).$ For each positive integer $k$, let $$E_k := \left\{x: \rho(x, Z) \in \left[2^{-k}, 2^{-k+1}\right]\right\}, ~~F_k = \left\{x: \rho(x, Z) \leq 2^{-k+1}\right\}.$$ Since the $(n-1)$-dimensional upper Minkowski content of $Z$ is zero, there is an integer $K$ such that whenever $k>K$, we have $\vol(E_k) \leq 2^{-k}.$ The H\"older condition implies that we may take $\eta = C^{-1}2^{-\gamma k}$ for some constant $C.$ Hence, $$||f||_{L^1(F_K)} \leq C\sum_{k=K}^\infty \left(\gamma k+\frac13(k-1) \right)2^{-k}\leq C\sum_{k=1}^\infty \left(\gamma k+\frac13(k-1) \right)2^{-k},$$ and this sum converges to a finite value depending only on $\gamma$. Since the $L^1$-norm of $f$ on the complement of $F_K$ is clearly bounded, we have proved the desired estimate.
\end{proof}

\begin{remark}\label{l1good}
    In fact, we have proved something slightly stronger, which is that in each $F_k,$ the $L^1$-norm is bounded by: $$||f||_{L^1(F_k)}\leq M\nu(k),$$ where $M$ is independent of $k$ and $\nu:\mathbb{N}\to\mathbb{R}$ is a function satisfying $\lim_{k\to \infty}\nu(k) = 0.$
\end{remark}

\begin{theorem}\label{currents}
    Suppose that $(A_j, \alpha_j, \beta_j, \epsilon_j)$ is a sequence of solutions to Equations \ref{rescaledholomsw} with respect to a K\"ahler parameter $(g, B, 2r\omega)$ for the line bundle $L\to X$, and converging to the limiting configuration $(A, \alpha, \beta, Z)$ in the sense of Theorem \ref{c-infty}. Then, taking a subsequence of the connections $A_j$ if necessary, the following holds in the sense of currents: $$\lim_{j\to\infty}\frac{\sqrt{-1}}{2\pi} F_{A_j}=\sigma+\frac12\sum_{i=1}^k q_iZ_i,$$ where $\sigma$ is a smooth closed real $(1,1)$-form that agrees with the form $\frac{\sqrt{-1}}{2\pi}F_A$ on $X - Z.$
\end{theorem}

\begin{proof}
    We will use the notation from the proof of Theorem \ref{c-infty}. For each connection $A_j$ we defined another connection $\widetilde{A_j} = e^{f_j/2} A_j$ (where we suppress the unitary gauge transformation without loss of generality.) The curvature satisfies: $$F_{A_j} = F_{\widetilde{A_j} } + \partial\overline{\partial} f_j.$$ Clearly, $F_{\widetilde{A_j}}$ converges in $C^\infty(X)$ to $F_{\widetilde{A}},$ which is some smooth imaginary 2-form on $X.$ We therefore wish to compute $\lim_{j\to\infty} \partial\overline{\partial} f_j.$

    \emph{Step 1: $L^1$ Convergence}. From Theorem \ref{c-infty}, some subsequence of the $f_j$ converges to a function $f:X-Z\to \mathbb{R}$ in $C^\infty_{\text{loc}}(X - Z).$ We will first show that in fact $f_j\to f$ in $L^1(X).$ Since differentiation is continuous in the weak topology of currents, this will imply $\partial\overline{\partial} f_j$ converges to $\partial\overline{\partial} f$ as currents. 
    
    Substituting the relation $F_{A_j} = F_{\widetilde{A_j}} + \partial\overline{\partial} f_j$ into the 4th equation of Equation (\ref{rescaledholomsw}), we see that $f_j$ satisfies the following differential equation: $$\epsilon_j\Delta f_j + |\widetilde{\beta_j}|^2e^{f_j}-|\widetilde{\alpha_j}|^2e^{-f_j}+w_j=0,$$ where $w_j:X\to\mathbb{R}$ is smooth and converges to 0 in $C^\infty(X)$ as $j\to \infty.$ Choosing $\gamma = \frac12,$ clearly $|\widetilde{\alpha_j}|^2_{C^{0,\frac12}(X)}$ and $|\widetilde{\beta_j}|^2_{C^{0,\frac12}(X)}$ are bounded independently of $j,$ so the hypotheses of Lemma \ref{l1} are satisfied, and the value of $M$ in the resulting bound, as well as the bounds in Remark \ref{l1good}, can be taken independently of $j$.

    Let $\delta>0$. Recalling the notation of Remark \ref{l1good}, there is some $J$ such that $\nu(j)<\frac{\delta}{4C}.$ Theorem \ref{c-infty} then guarantees a $J',$ depending on $J,$ such that $$||f_j - f||_{L^1(X - F_J)}<\frac\delta2$$ for all $j\geq J'.$ Remark \ref{l1good} together with Fatou's Lemma show that for all $j\geq J$ we have: $$||f_j-f||_{L^1(F_J)}\leq ||f_j||_{L^1(F_J)}+||f||_{L^1(F_J)}\leq 2C\nu(J)<\frac\delta2.$$ Combining these estimates, we get $||f_j-f||_{L^1(X)}< \delta$ for all $j>\max(J, J'),$ so since $\delta$ was chosen arbitrarily this suffices to prove that $f_j\to f$ in $L^1(X).$

    \emph{Step 2: Calculation of the current}. The function $f:X- Z\to\mathbb{R}$ clearly satisfies $f = \log|\widetilde{\alpha}|-\log |\widetilde\beta|.$ Since $\widetilde\alpha, \widetilde\beta$ are holomorphic, the current  $\partial\overline{\partial} f$ is supported on the zero sets of $\widetilde\alpha$ and $\widetilde\beta.$ Since $e^{-f/2}|\widetilde\alpha|=|\alpha|,$ and $|\alpha|$ is continuous, it is clear that the vanishing set of $\alpha'$ contains that of $\alpha,$ and likewise for $\widetilde\beta.$ Hence, the zero sets of $\widetilde\alpha$ and $\widetilde\beta$ consist of isolated points as well as codimension 1 irreducible components $Z_1, \dots, Z_M$, where $M\geq k$ and $Z_1,\dots, Z_k$ are the irreducible codimension 1 components of $Z.$ Recall that $|\widetilde\beta| = |\widetilde\gamma|,$ and $\widetilde\gamma$ is a holomorphic section of $E^*\otimes L^{*}\otimes K.$ Hence, there exist meromorphic functions $s_\alpha$ and $s_\beta$ with divisors supported on $Z_1\cup\dots\cup Z_M$ such that $s_\alpha^{-1}\widetilde\alpha$ and $s_\beta^{-1}\widetilde\beta$ vanish at only isolated points. We may therefore write: $$f = \log|s_\alpha|-\log |s_\beta|+ \log|s_\alpha^{-1}\widetilde\alpha|-\log |s_\beta^{-1}\widetilde\beta|.$$ Hence, the Poincar\'e-Lelong formula (see for instance \cite[p. 188]{GH}) tells us $$\partial\overline{\partial} f = T-2\pi\sqrt{-1}\sigma$$ where $$\sigma := \partial\overline{\partial}\log|s_\alpha^{-1}\widetilde\alpha|-\partial\overline{\partial}\log |s_\beta^{-1}\widetilde\beta|$$ is smooth except possibly at isolated zeroes of $s_\alpha^{-1}\widetilde\alpha, s_\beta^{-1}\widetilde\beta,$ and where $\frac{\sqrt{-1}}{\pi}T$ is the current associated to a divisor supported on $Z_1\cup\dots\cup Z_M$. The $C^\infty_{\text{loc}}(X - Z)$ convergence of $A_i$ provides the claim that $\sigma$ agrees with the smooth form $\frac{\sqrt{-1}}{2\pi}F_A$ on $X - Z.$ 

    \emph{Step 3: Showing $\sigma$ is a smooth form.} Suppose $z_0\in X$ is an isolated zero of $s_\alpha^{-1}\widetilde\alpha$ or $s_\beta^{-1}\widetilde\beta$, and let $U$ be a small closed neighborhood of $z_0$, that we may identify with a ball in $\mathbb{C}^2$ centered at the origin. We choose $U$ such that $z_0$ is the only zero of $s_\alpha^{-1}\widetilde\alpha$ or $s_\beta^{-1}\widetilde\beta$ in $U.$ In a trivialization of $E\otimes L$ over $U$, we will define some notation: $a:=s_\alpha^{-1}\widetilde\alpha = (\alpha_1, \alpha_2, \dots, \alpha_N)$, and $b:= s_\beta^{-1}\widetilde\beta = (\beta_1, \beta_2, \dots, \beta_N)$. Defining $G:X\to\mathbb{R}$ as $$G := \log \frac{|\alpha_1|^2+\dots+|\alpha_N|^2}{|\beta_1|^2+\dots+|\beta_N|^2} = \log \frac{|a|^2}{|b|^2},$$ we have the equality of currents $$\sigma = \frac12\partial\overline{\partial}G.$$ It is clear that $G$ is smooth wherever it is bounded; therefore, if we can show $G$ is bounded on $U$, then $\sigma$ is in fact smooth on $U$. 
    
    %We will do this in two sub-steps. \emph{Step 3A: Showing $\partial\overline{\partial}G$ is a weak derivative.} We will first show that on $U,$ the current $\partial\overline{\partial}G$ is in fact given by integration against an $L^1$-function. Note that the function $\partial\overline{\partial}G\in $ Recall that a current $\Theta$ is \emph{normal} if $\Theta$ and $d\Theta$ are both currents of order zero, i.e. each coordinate function is given by a complex measure (see \cite[\sectionmark 2.C]{Demailly} for instance.) To do this, it is sufficient to show an $L^1_2(U)$ bound on $G.$ Upon identifying $U$ with a ball in $\mathbb{C}^2,$ taking $z_0$ to the origin, we let $z, w$ be holomorphic coordinates and $r$ the radial function. Since the functions $\alpha_i, \beta_i^*$ are holomorphic, WEIERSTRASS PREPARATION there are some constants $C, N, M$ such that for each $k\in\{0,1,2\}$: $$C^{-1}r^{N-k}\leq\left|\nabla^k|a|\right|\leq Cr^{M-k},~C^{-1}r^{N-k}\leq\left|\nabla^k|b|\right|\leq Cr^{M-k},$$ where $\nabla^k$ indicates the $k$th covariant derivative, and $\nabla^0$ indicates not taking a derivative. Hence, for each pair of real coordinates $e_i, e_j,$ we have: $$\left|\nabla_i \log |a|\right| = \frac{\left|\nabla_i|a|\right|}{|a|}\leq C^2r$$v$$\left|\nabla_i\nabla_j \log |s_\alpha^{-1}\widetilde\alpha|\right| =  $$

    \emph{Step 3A: Showing $\sigma$ is a weak derivative.} We know that in the sense of distributions, $\sigma = \partial\overline{\partial}G$; we will show that this is true in the sense of weak derivatives as well. 

    We first show that there exists an $L^1(U)$ differential form $F$ which satisfies $F = \partial\overline{\partial}G$ in the sense of weak derivatives. To do this, note $G = \log |a|^2 - \log |b|^2$ is the difference of two plurisubharmonic functions, so by \cite[Section 2.C and Remark 3.4.A, Chapter 3]{Demailly}, $\sigma = \partial\overline{\partial} G$ is a current of order zero (i.e. each coordinate function is given by a complex measure) on $U$ and the mass norm $|\sigma|_U$ is bounded. Therefore, the mass norm of $\mathds{1}_{U - \{z_0\}} \sigma$ is also bounded, but since $G$ is smooth on $U - \{z_0\},$ this latter quantity is given by $|\partial\overline{\partial} G|_{L^1(U)}.$ So defining $F$ to be 0 at $z_0$ and as $F := \partial\overline{\partial} G$ on $U - \{z_0\}$, we see $F$ is an $L^1$ weak solution of $F = \partial\overline{\partial} G$ on $U$, sufficing for the claim.
    
    By definition, $F = \mathds{1}_{U - \{z_0\}} \sigma$, so we must now show $\sigma = \mathds{1}_{U - \{z_0\}} \sigma.$ Note that $F$ is closed as a current; this follows either from the El Mir Extension Theorem \cite[Corollary 2.11, Chapter 3]{Demailly} or from the Harvey-Polking Theorem \cite[Theorem 2.5]{Harvey-Polking}.
    
    %On $U - U\cap Z,$ we have the equality of currents $\partial\overline{\partial}G = \frac{\sqrt{-1}}{2\pi} F_A$ from Step 2. Lemma \ref{fabounded} provides a $C^0(U)$ bound, and thus an $L^1(U)$ bound, on $\frac{\sqrt{-1}}{2\pi} F_A$, so it defines a current $\widetilde{F}$ on all of $U$ by integration. By the Harvey-Polking Theorem \cite[Theorem 2.5]{Harvey-Polking}, the current $\widetilde{F}$ is also closed. Since the current $\widetilde{F}$ is $L^1$ and closed, it is clearly normal.
    
   Recall that a current $\Theta$ is \emph{normal} if $\Theta$ and $d\Theta$ are both currents of order zero (see \cite[Section 2.C, Chapter 3]{Demailly} for instance.) Being both closed and the difference of two positive currents, $\sigma =\partial\overline{\partial}G$ and $F = \mathds{1}_{U - \{z_0\}} \sigma$ are both normal. By \cite[Corollary 2.11, Chapter 3]{Demailly}, two normal $(p,p)$-currents whose difference is supported in an analytic subset of dimension $<p$ are equal. Since $\sigma- F$ is supported at $z_0,$ the equality of currents $\sigma = F$ must therefore hold on all of $U.$
   
   %Since has measure zero, $$\widetilde{F} = \mathds{1}_{U - U\cap Z}\frac{\sqrt{-1}}{2\pi} F_A = \mathds{1}_{U - U\cap Z}\partial\overline{\partial}G,$$ and $G$ is smooth on $U - U\cap Z$, we must have that $\sigma =\partial\overline{\partial}G$ in the sense of weak derivatives. Moreover, since $\frac{\sqrt{-1}}{2\pi} F_A$ satisfies a $C^0(U)$ bound, clearly $\sigma \in L^p(U)$ for each $p.$
    
    \emph{Step 3B: Showing a $C^0(U)$ bound on $G$.} We now wish to use elliptic estimates and the Taubes theorem to show a $C^0$ bound on $G.$

    We first claim $G\in L^p(U)$ for each $p.$ Since $G = \log |a|^2 - \log |b|^2,$ and each term is the logarithm of a real analytic function vanishing only at one point in $U,$ that $G\in L^p(U)$ follows from considering the Taylor expansions of these functions, and the integrability of $(\log r)^p$ on the unit ball in $\mathbb{R}^4$ where $r$ is the radial coordinate. 

    We now claim $\partial\overline{\partial}G$ is in $L^p(U)$ for each $p<\infty.$ From Step 2, we know that on $U - U\cap Z,$ the $L^1$ differential form $\partial\overline{\partial}G$ agrees with the form $\frac{\sqrt{-1}}{2\pi} F_A.$ By Lemma \ref{fabounded}, the form $\frac{\sqrt{-1}}{2\pi} F_A$ satisfies a $C^0(U- U\cap Z)$ bound, hence a $L^p(U)$ bound for each $p.$ Since $Z$ has measure zero (see Lemma \ref{Zanalytic}), we conclude $\partial\overline{\partial}G$ must also satisfy an $L^p(U)$ bound for each $p.$ Applying the contraction with the K\"ahler form $\Lambda,$ we get an $L^p(U)$ bound on $\Lambda \sigma$ for each $p$ as well. 
    
    Now, we put these bounds together with an elliptic estimate. By the K\"ahler identities $\Lambda\partial\overline{\partial} = i\Delta,$ which is elliptic, and by Step 3A, we know $-i\Lambda \sigma = \Delta G$ in the sense of weak derivatives. Hence, the Calder\'on-Zygmund elliptic estimate gives us: $$||G||_{L^4_2(U)}\leq C\left(||\Delta G||_{L^4(U)}+||G||_{L^4(U)}\right),$$ hence there is a bound on $||G||_{L^4_2(U)}.$ Since $U$ has real dimension 4, there is a Sobolev embedding $C^0(U)\hookrightarrow L^4_2(U),$ and we conclude that $||G||_{C^0(U)}$ is bounded. By the discussion before Step 3A, this fact implies that $\sigma$ is represented by a smooth form on all of $U.$ Repeating this argument for each isolated zero of $s_\alpha^{-1}\widetilde\alpha$ or $s_\beta^{-1}\widetilde\beta$, we show $\sigma$ is in fact represented by a smooth form on all of $X.$
    
    \emph{Step 4: Calculating the multiplicities of $\frac{\sqrt{-1}}{\pi}T$.} What remains to be shown is that the multiplicities $\frac12q_j$ agree with those of $\frac{\sqrt{-1}}{2\pi}T$. We know that $\frac{\sqrt{-1}}{2\pi}T$ is the current associated to some divisor. Hence, this claim follows from the fact that, in the notation of Lemma \ref{q}, $$\int_{\partial D_r}a = \int_{D_r} (2F_A - F_{\nabla_0}) = \lim_{i\to\infty} \int_{D_r} (2F_{A_i} - F_{\nabla_0}) = \lim_{i\to\infty} \int_{D_r} 2F_{A_i}.$$
\end{proof}

\subsection{Limiting configurations when $N =2$}

We can say slightly more when the rank of the auxiliary bundle $E$ is 2. The key input in the special case $N=2$ is the following result of Taubes, which we state in a very special case, the K\"ahler case, since this is what we will need.

\begin{theorem}[special case of {\cite[Proposition 1.2]{Taubes16}}]\label{taubes4}
    With $X, B, \omega$ as above, if $E$ is an $SU(2)$-bundle, then for any limiting configuration $(A, (\alpha, \beta), Z),$ there is in addition a unitary isomorphism of bundles $L^{\otimes 2}\cong K_X$ on $X-Z$ taking $A\otimes A$ to $\nabla_0.$ 
\end{theorem}

\begin{remark}
    The result \cite[Proposition 1.2]{Taubes16} is far more general than what is stated here: in fact, it applies to equations more general than the multi-spinor Seiberg-Witten equations considered here. But, even the statement for these equations holds for a Riemannian manifold with a topological condition on the spinor bundles. 
\end{remark}

Theorem \ref{currents} tells us the equality of smooth currents $\sigma = \frac1{4\pi}F_{\nabla_0}$ on $X - Z.$ Since both sides of this equation are smooth, this equality must hold everywhere. Hence, the cohomology classes of these closed (1,1)-currents must be equal. The Chern-Weil formula then immediately tells us:

\begin{corollary}
    When $N = 2,$ in $H^2(X; \mathbb{R})$ we have the following equality: $$2c_1(L) - c_1(K_X) = \sum_{j=1}^k q_j[Z_j].$$
\end{corollary}

\section{Computations for K\"ahler surfaces}

In this section, we compute the multi-spinor Seiberg-Witten function and the moduli space $\mathcal{M}(X, \mathbf p)$ in a number of examples. In this section, we will always assume that $(X, g)$ is K\"ahler, and that the parameter $\mathbf{p}$ is \emph{good}:

\begin{definition}
    Let $X$ be a smooth, closed, oriented 4-manifold. A parameter $\mathbf{p}\in \mathscr{P}$ is \emph{good} if $\mathbf{p} = (g, B, 2r\omega)$ for a K\"ahler metric $g$ and associated K\"ahler form $\omega,$ if $r$ sufficiently large, and $B$ is WUASD and has been chosen such that $\mathcal{M}$ is Zariski smooth.
\end{definition}

We will moreover assume that $b_1(X) = 0$ since this drastically simplifies our computations.

Then, the virtual (complex) dimension of $\mathcal{M}$ is given by the index of the complex $\mathcal{S}_\mathbb C$: $$\text{virtual}\dim_\mathbb{C}\mathcal{M} = \chi(E\otimes L) - \chi(\mathscr{O}_X),$$ which we could alternatively express entirely in terms of the topological data of $E, L, X$ using the Hirzebruch-Riemann-Roch Theorem.

%One can form the trivial bundle of this complex over the configuration space $\mathcal{A}(L)\times \Omega^0(W^+\otimes E).$ The differentials are equivariant with respect to the action of the gauge group $\mathcal{G}$ by multiplication on the spinor summands of the complex and trivially on the remaining summands. Hence, this complex of bundles descends to a complex of bundles on the space $\mathscr{B} = \left(\mathcal{A}(L)\times \Omega^0(W^+\otimes E)\right)/\mathcal{G},$ which we denote $\mathcal{C}_\mathbb{R}.$

We will use the notation: \begin{itemize}
    \item $v$ is the virtual complex dimension of $\mathcal{M}.$
    \item $a$ is the actual complex dimension of $\mathcal{M}.$
\end{itemize}

Since we assume that $b_1(X) = 0,$ the Picard group of $X$ is discrete, and moreover by the exponential sheaf sequence we know that the map $c_1:\text{Pic}(X)\to H^2(X;\mathbb{Z})$ is injective, so for each topological line bundle $L,$ there is at most one holomorphic structure $\mathscr{L}$ on $L,$ and the moduli space $\mathcal{M}$ is either empty or is Zariski smooth and precisely the projective space $\mathbb{P}H^0(\mathscr{E}\otimes \mathscr{L}).$ On such a moduli space, we wish to compute the integral in Theorem \ref{zariski-smooth}. This will require us to have a formula for the universal bundle $\mu,$ as well as for $e(\mathfrak{O}),$ which follows from formulas for the Chern classes of $\ind \mathcal{S}_\mathbb{C}$ and $T\mathcal{M}$ by Equation \ref{e(o)}.

Clearly, on each such projective space $\mathbb{P}H^0(\mathscr{E}\otimes \mathscr{L})$, the universal bundle $\mu$ restricts to $\mathscr{O}(1),$ and the index bundle of the complex $\mathcal{S}$ is, stably, $$\ind \mathcal{S} =[H^1(\mathscr{E}\otimes\mathscr{L})\otimes \mathscr{O}(1)] -[H^2(\mathscr{E}\otimes\mathscr{L})\otimes \mathscr{O}(1)]-[H^0(\mathscr{E}\otimes\mathscr{L})\otimes \mathscr{O}(1)] \in K(\mathcal{M}).$$ In this case, $a = h^0(\mathscr{E}\otimes\mathscr{L}) - 1.$ Denoting $c_1(\mathscr{O}(1))$ by $x\in H^2(\mathbb{P}^a; \mathbb{Z}),$ we recall that $c(T\mathbb{P}^a) = (1+x)^{a+1}.$ To simplify the notation, let $$m := h^2(\mathscr{E}\otimes\mathscr{L}) - h^1(\mathscr{E}\otimes\mathscr{L}) = \chi(E\otimes L)-h^0(\mathscr{E}\otimes\mathscr{L}).$$ Then, Equation \ref{e(o)} gives us: $$e(\mathfrak{O}) =\left[(1+x)^{-m}\right]_{h^2(\mathscr{O}_X)-m},$$ where the right-hand side is interpreted as 0 if $h^2(\mathscr{O}_X)-m<0.$ Putting the formulas together we have: $$SW_E(X, \mathbf p) = \int_{\mathbb{P}^a}\left[(1+x)^{-m}\right]_{h^2(\mathscr{O}_X)-m}x^{v}.$$ 

\begin{corollary}
    For a good parameter $\mathbf{p}$, the Seiberg-Witten function $SW_E(X, \mathbf p) = 0$ unless $v\geq0$ and $h^2(\mathscr{O}_X)-m\geq 0.$
\end{corollary}

\subsection{A Geometric Inequality} It will be useful (and heartening) to know that whenever the moduli space is non-empty that we have $a\geq v,$ and in the $b_1(X) = 0$ case this is implied by the statement: $$h^2(\mathscr{O}_X)\geq h^2(\mathscr{E}\otimes\mathscr{L})~\text{whenever}~h^0(\mathscr{E}\otimes\mathscr{L})>0.$$ In the strongly unobstructed case, we may prove this directly (i.e. without gauge theory) as follows:

\begin{lemma}\label{a-v}
    Let $\mathscr{F}$ be a holomorphic vector bundle on a K\"ahler surface $X$ (not necessarily trivial determinant). Suppose that $h^2(\End_0 \mathscr{F}) = 0$ and $h^0(\mathscr{F})>0.$ Then, $$h^2(\mathscr{O}_X)\geq h^2(\mathscr{F}).$$
\end{lemma}

\begin{proof}
    Consider the holomorphic map of complex affine spaces $$S: H^0(\mathscr{F})\oplus H^0(\mathscr{F}^*\otimes K)\to H^0(\End \mathscr{F}\otimes K)$$ given by: $S(\alpha, \beta) := \alpha\otimes \beta.$ First, we claim that $S^{-1}(0) = H^0(\mathscr{F})\cup H^0(\mathscr{F}^*\otimes K).$ For, if $(\alpha, \beta)\in S^{-1}(0),$ then $\alpha\otimes\beta = 0,$ so at each point $p\in X,$ one of $\alpha(p)$ or $\beta(p)$ must vanish. Both $\alpha$ and $\beta$ are global solutions to an elliptic PDE on $X,$ so by the unique continuation theorem, one of $\alpha$ or $\beta$ must be identically zero. Conversely, if one of $\alpha$ or $\beta$ is identically zero then clearly $S(\alpha, \beta) = 0.$ This suffices for the claim.

    The dimension of fibers theorem allows us to bound the dimension of any irreducible component of a fiber of a map from below. In particular, we must therefore have that $$h^0(\mathscr{F}) \geq h^0(\mathscr{F})+h^2(\mathscr{F}) - \dim \text{image}(S)\geq h^0(\mathscr{F})+h^2(\mathscr{F}) - h^0(\text{End}\mathscr{F}\otimes K).$$
    
    Since $h^0_X(\End_0 \mathscr{F}\otimes K) = h^2_X(\End_0 \mathscr{F}) = 0,$ we conclude that the trace map $$\tr:H^0(\End \mathscr{F}\otimes K)\to H^0(K)$$ is an isomorphism. Combining this with our previous inequality gives $h^2(\mathscr{O}_X)\geq h^2(\mathscr{F}),$ as desired. 
\end{proof}

Note that by the formula \ref{zariski-smooth}, the value of the Seiberg-Witten function is 0 unless $a\geq v,$ a fact we will use often in our computations below.

\subsection{The case $h^2(\mathscr{O}_X) = 0$}

We will first concentrate on the case $h^2(\mathscr{O}_X) = 0$, in which the virtual dimension is given by: $$v = h^0(\mathscr{E}\otimes \mathscr{L}) - 1 + m,$$ and when the moduli space is non-empty the actual dimension is given by: $$a = h^0(\mathscr{E}\otimes\mathscr{L}) - 1$$

Assume we have chosen $B$ such that $h^2(\End \mathscr{E}) = 0.$

When $h^0(\mathscr{E}\otimes\mathscr{L})= 0,$ the moduli space is empty and therefore the relevant Seiberg-Witten function is zero as well.

When $h^0(\mathscr{E}\otimes\mathscr{L})>0,$ we know $m\leq 0$ by Lemma \ref{a-v}. Hence, the Euler class expression evaluates to: $$e(\mathfrak{O}) =\left[(1+x)^{-m}\right]_{-m} = x^{-m},$$ and the Seiberg-Witten function is 1 whenever the moduli space is non-empty and $v\geq 0$: $$ \begin{cases}
h^0(\mathscr{E}\otimes\mathscr{L})\geq \chi(\mathscr{E}\otimes\mathscr{L})>0 & SW_E(X, \mathbf{p}) = 1 \\
\text{otherwise} & SW_E(X, \mathbf{p}) = 0 \\
\end{cases}$$ 

When $\mathscr{M}$ is a positive line bundle, for sufficiently large $k$ the bundle $\mathscr{E}\otimes \mathscr{M}^k$ has vanishing higher cohomology, so we immediately have the following example:

\begin{corollary}
    If $X$ is algebraic with $h^1(\mathscr{O}_X)=h^2(\mathscr{O}_X) = 0$ (for instance, if $X$ is $\mathbb{P}^2,$ a Hirzebruch surface, or a blow-up thereof) then there exists a good parameter $\mathbf{p}$ such that $SW_E(X, \mathbf{p}) = 1$ for infinitely many topological line bundles.
\end{corollary}

This shows that the Seiberg-Witten function can be supported in infinitely many degrees in the same chamber, which is a markedly different behavior than the count of solutions to the 1-spinor Seiberg-Witten equations, and also in stark contrast to the behavior of the multi-spinor Seiberg-Witten invariant for Riemann surfaces \cite{Doan, Thakar}.

\subsection{The case $h^2(\mathscr{O}_X)>0$}
In the case $h^2(\mathscr{O}_X)>0,$ we see immediately from our formula: $$e(\mathfrak{O}) =\left[(1+x)^{-m}\right]_{h^2(\mathscr{O}_X)-m}$$ that $e(\mathfrak{O}) = 0$ whenever $m\leq 0.$

So, the Seiberg-Witten function must be zero unless $m>0$, and of course $v\geq 0, a\geq0.$ Under these inequalities, the Taylor expansion for $(1+x)^{-m}$ allows us to compute that: $$SW_E(X, \mathbf{p}) = (-1)^{h^2(\mathscr{O}_X) - m}{{h^2(\mathscr{O}_X) - 1}\choose{h^2(\mathscr{O}_X) - m}} = (-1)^{h^2(\mathscr{O}_X) - m}{{h^2(\mathscr{O}_X) - 1}\choose{m - 1}},$$ and $SW_E(X, \mathbf{p}) = 0$ otherwise. Note that when $a\geq 0, v\geq 0$, then Lemma \ref{a-v} guarantees $h^2(\mathscr{O}_X)-m\geq0$ so that in fact $SW_E(X, \mathbf{p})\neq0$ whenever $m>0.$ We rewrite these inequalities algebro-geometrically and package them in the following theorem:

\begin{theorem}\label{agsw}
    Let $X$ be a K\"ahler surface with $h^1(\mathscr{O}_X) = 0, h^2(\mathscr{O}_X) >0.$ If $$\chi(E\otimes L)> h^2(\mathscr{O}_X)~\text{and}~\chi(E\otimes L)>h^0(\mathscr{E}\otimes \mathscr{L})>0,$$ then $$SW_E(X, \mathbf{p}) = (-1)^{h^2(\mathscr{O}_X) - m}{{h^2(\mathscr{O}_X) - 1}\choose{h^2(\mathscr{O}_X) - m}}\neq0,$$ and otherwise: $$SW_E(X, \mathbf{p}) = 0.$$
\end{theorem}

The inequalities $\chi(E\otimes L)> h^2(\mathscr{O}_X)$ and $\chi(E\otimes L)>h^0(\mathscr{E}\otimes \mathscr{L})>0,$ in turn require that: $$h^0(\mathscr{E}\otimes \mathscr{L})>0, ~h^0(\mathscr{E}^*\otimes \mathscr{L}^*\otimes K_X)=h^2(\mathscr{E}\otimes \mathscr{L})>0,$$ so since we assume $\mathscr{E}$ is stable, and with slope 0, we must have: 

\begin{lemma}
    Let $X$ be a K\"ahler surface with $h^1(\mathscr{O}_X) = 0, h^2(\mathscr{O}_X) >0.$ For a good parameter $\mathbf{p}$, if $SW_E(X, \mathbf{p})\neq 0$ we must have: $$0<\deg \mathscr{L} < \deg K_X.$$
\end{lemma} 

As an immediate corollary we have the following computation:

\begin{corollary}
    For a good parameter $\mathbf{p}$, we have $SW_E(X, \mathbf{p}) = 0$ for $X$ a $K3$ surface, an Enriques surface, or a smooth hypersurface of degree $5$ in $\mathbb{CP}^3$ with the Fubini-Study metric.
\end{corollary}

We may compute the Seiberg-Witten function in terms of the topology of $E,$ as well. Let $k$ denote the quantity: $$k := \int_X c_2(E).$$ Then, the Hirzebruch-Riemann-Roch Theorem tells us that: $$\chi(E\otimes L) = N\chi(L) - k,$$ so the condition $v\geq 0$ requires: $k<N\chi(L) -p_g.$ 

We will now prove a vanishing theorem in the $N = 2$ case for several general type surfaces.

\begin{theorem}
    Let $X$ be a K\"ahler surface with $h^1(\mathscr{O}_X) = 0$ and $h^2(\mathscr{O}_X)>0.$ If $$h^0(K_X\otimes \mathscr{L}^{-2})+h^0(K_X^{-1}\otimes \mathscr{L}^{2})>0$$ then $SW_E(X, \mathbf{p}) = 0$ whenever $N = 2$ and $\mathbf{p}$ is good.
\end{theorem}

\begin{proof}
    Suppose $h^0(K_X\otimes \mathscr{L}^{-2})> 0$, and let $t$ be a non-trivial section. If, in fact, $h^0(K_X^{-1}\otimes \mathscr{L}^{2})>0,$ instead, then in the below argument replace $\mathscr{L}$ by $K_X\otimes \mathscr{L}^{-1}.$
    
    Suppose $SW_E(X, \mathbf{p}) \neq 0$ for some good parameter corresponding to a holomorphic bundle $\mathscr{E}$. Then, $\mathscr{E}\otimes \mathscr{L}$ has a non-trivial section $s$. Hence, $s\otimes t$ is a section of $\mathscr{E}\otimes\mathscr{L}^{-1}\otimes K_X$. The tensor product $s\otimes (s\otimes t)$ is therefore a non-trivial section of $H^0(\End_0 \mathscr{E}\otimes K_X)$ which satisfies $\tr(s\otimes(s\otimes t)) = 0.$ Hence, $\mathscr{E}$ is not weakly unobstructed, contradicting that $\mathbf{p}$ was good.
\end{proof}

\begin{corollary}
    When $N=2$ and $X$ is a general degree $d$ hypersurface in $\mathbb{CP}^3,$ $d\geq 4,$ then $SW_E(X, \mathbf{p}) = 0$ for all good parameters $\mathbf{p}.$
\end{corollary}

\begin{proof}
    By the Noether-Lefschetz Theorem, a general degree $d$ hypersurface $X\subset \mathbb{CP}^3,$ $d\geq4$, has $\text{Pic}(X) = \mathbb{Z}\langle\mathscr{O}(1)\rangle.$ Moreover, the canonical bundle is $K_X = \mathscr{O}(d - 4)$ by the adjunction formula. Since $h^0(\mathscr{O}(k))>0$ whenever $k\geq0,$ any holomorphic line bundle either has a section, or its dual has a section. Thus: $$h^0(K_X\otimes \mathscr{L}^{-2})+h^0(K_X^{-1}\otimes \mathscr{L}^{2})>0.$$
\end{proof}

The local invariance of the Seiberg-Witten function, Theorem \ref{local-invariance}, will imply the following purely algebro-geometric theorem:

\begin{theorem}\label{algebraicgeometry}
    Let $X$ be a K\"ahler surface with $h^1(\mathscr{O}_X) = 0, h^2(\mathscr{O}_X)>0.$ Let $\mathcal{N}$ be an irreducible component of the moduli space of stable vector bundles on $X$ with holomorphically trivial determinant. Let $\mathscr{L}\to X$ be a holomorphic line bundle.
    
    If there is some $\mathscr{F}_0\in\mathcal{N}$ satisfying the following three conditions: \begin{itemize} 
        \item $H^2(\End_0 \mathscr{F}_0) = 0$
        \item $\chi(\mathscr{F}_0\otimes \mathscr{L})>h^0(\mathscr{F}_0\otimes \mathscr{L})>0$ 
        \item $\chi(\mathscr{F}_0\otimes \mathscr{L})>h^2(\mathscr{O}_X)$
    \end{itemize} then for any $\mathscr{F}\in\mathcal{N}$ with $H^2(\End_0 \mathscr{F}) = 0$, either $$h^0(\mathscr{F}\otimes\mathscr{L}) = h^0(\mathscr{F}_0\otimes \mathscr{L})$$ or $$h^0(\mathscr{F}\otimes \mathscr{L})=2\chi(\mathscr{F}\otimes \mathscr{L}) - h^2(\mathscr{O}_X)-h^0(\mathscr{F}_0) .$$
\end{theorem}

\begin{proof}
    Let $\mathcal{N}^*\subset \mathcal{N}$ be the subset consisting of those stable vector bundles $\mathscr{F}$ with $H^2(\End_0 \mathscr{F}) = 0$. By the Grauert Semicontinuity Theorem $\mathcal{N}^*\subset \mathcal{N}$ is the complement of an analytic subvariety, so since it is non-empty by assumption, it is path-connected. Let $\gamma:[0,1]\to \mathcal{N}^*$ denote a path between $\mathscr{F}_0$ and $\mathscr{F}$. Let $E$ denote the underlying smooth vector bundle. We claim there exists a path $\widetilde{\gamma}:[0,1]\to \mathscr{A}(E)$ such that $\widetilde{\gamma}(t)$ is an unobstructed Hermitian-Yang-Mills connection corresponding to the holomorphic structure $\gamma(t).$ We show the theorem now assuming the claim. For, each parameter on the path of parameters $\mathbf{p}(t) = (g, \widetilde{\gamma}(t), 2r\omega)$ is good, so $\mathbf{p}([0,1])\subset \mathscr{U}_{cpt}$ by Theorem \ref{strong_compact}. By Theorem \ref{local-invariance}, $SW_E(X, \mathbf{p}(t))$ is constant, hence $SW_E(X, \mathbf{p}(0))= SW_E(X, \mathbf{p}(1))$. The formula in Theorem \ref{agsw} then implies the conclusion of the above theorem.

    We now address the claim. There are several ways to see this; we follow \cite{Itoh-Nakajima}. Let $\overline{\mathcal{N}}$ denote the moduli space of all holomorphic structures $\mathscr{F}$ on $E$ which satisfy $h^0(\End_0\mathscr{F}) = h^2(\End_0\mathscr{F}) = 0$; this is a not-necessarily-Hausdorff manifold \cite[Theorem 2.10]{Itoh-Nakajima}, and by stability, $\mathcal{N}^*\subset \overline{\mathcal{N}}.$
    
    The space $\mathcal{Y}$ of irreducible, unobstructed Hermitian-Yang-Mills connections on $E$ modulo unitary gauge is a complex manifold, and the map $\mathcal{Y}\to \overline{\mathcal{N}}$ assigning $[B]\in \mathcal{Y}$ to the holomorphic structure given by $\overline{\partial}_B$ is open \cite[Theorem 2.13]{Itoh-Nakajima}. This map is onto $\mathcal{N}^*$ (see, for instance \cite[Chapter 6]{DK}.) Therefore, there is a lift $\gamma':[0,1]\to\mathcal{Y}$ of $\gamma$ by the map $\mathcal{Y}\to \overline{\mathcal{N}}$. Now, $\mathscr{Y}$ is the quotient of a subset $\widetilde{Y}\subset \mathscr{A}(E)$ by the action of the unitary gauge group $\mathscr{G}$; denote by $\pi:\widetilde{Y}\to Y$ the projection. Since $\gamma'([0,1])$ consists only of irreducible connections, the action of $\mathscr{G}$ on $\pi^{-1}(\gamma'([0,1]))$ is free, so it is clear that there exists a lift $\widetilde{\gamma}(t)$ of $\gamma'(t)$ to $\widetilde{\mathcal{Y}}\subset \mathscr{A}(E),$ which suffices for the claim.
\end{proof}

%DO I NEED TRIVIAL DETERMINANT

\begin{remark}
    It is also interesting, however much more difficult, to compute the values of multi-spinor Seiberg-Witten invariants for K\"ahler surfaces with $b_1(X) > 0.$ For these surfaces, with sufficient knowledge of the cup product structure on the cohomology ring, we could use the families Atiyah-Singer index theorem analogously to the computation in \cite{Thakar} to accomplish this, although in most cases this computation would be highly unwieldy.
\end{remark}

\bibliographystyle{alpha}
\bibliography{main}

\end{document}